\documentclass[11pt,a4j]{article}

\usepackage{lineno,hyperref}
 \usepackage{amsmath,amssymb,amsfonts}
 \usepackage{amsthm}
 \usepackage{latexsym}
 \usepackage{array}

\theoremstyle{definition}
\newtheorem{theorem}{Theorem}[section]

\newtheorem{prop}[theorem]{Proposition}
\newtheorem{lemma}[theorem]{Lemma}
\newtheorem{rem}{Remark}
\newtheorem{example}{Example}
\newtheorem{defn}{Definition}

\allowdisplaybreaks[4] 

\title{Hom-Poisson-Nijenhuis structures on Hom-Lie algebroids and Hom-Dirac structures on Hom-Courant algebroids}
\author{Tomoya Nakamura  \thanks{Department of Applied Mathematics, Waseda University, \textup{3-4-1}, Okubo, Shinjuku-ku, Tokyo, Japan}
 \\ email: \href{mailto:x-haze@ruri.waseda.jp}{x-haze@ruri.waseda.jp}}
\date{\today}

\begin{document}

  \maketitle
  
  \begin{abstract}
  In this paper, we develop the theory of Hom-Lie algebroids, Hom-Lie bialgebroids and Hom-Courant algebroids introduced by Cai, Liu and Sheng \cite{CLS}. Specifically, we introduce the notions of Hom-Poisson, Hom-Nijenhuis and Hom-Poisson-Nijenhuis structures on a Hom-Lie algebroid and the notion of Hom-Dirac structures on a Hom-Courant algebroid. We show that these structures satisfy similar properties to structures non ``Hom-''version. For example, there exists the hierarchy of a Hom-Poisson-Nijenhuis structure and we have a relation between Hom-Dirac structures and Maurer-Cartan type equation. Moreover we show that there exists a one-to-one correspondence between the pairs consisting of a Poisson structure on $M$ and a Poisson isomorphism for it, and Hom-Poisson structures on $M$ introduced in \cite{CLS}. 
  \end{abstract}

\section{Introduction}
Hom-Lie algebras are introduced in the study of $\sigma $-derivations of an associative algebra in \cite{HLS}. As a geometric generalization of Hom-Lie algebras, Laurent-Gengoux and Tales introduced the concept of Hom-Lie algebroids \cite{LT} and Cai, Liu and Sheng improved them \cite{CLS}. Hom-Lie algebroids are also a generalization of Lie algebroids. Lie algebroids themselves are a generalization of both Lie algebras and the tangent bundles (for more details, see \cite{M}). In \cite{CLS}, Cai, Liu and Sheng developed a Hom-Lie algebroid theory and defined several notions related to Hom-Lie algebroids, that is, Hom-Poisson structures (HPS) on a manifold $M$, Hom-Lie bialgebroids and Hom-Courant algebroids. These notions are generalizations of Poisson structures (PS) on $M$, Lie bialgebroids and Courant algebroids, respectively, and satisfy similar properties to each corresponding property in the Lie algebroid theory. 

First, one of the aims of this paper is to study HPSs. While the authors in \cite{CLS} introduced the notion of HPSs on $M$, we introduce the notion of HPSs on arbitrary Hom-Lie algebroid in this paper. Then the former is a HPS on $\varphi^!TM$. We show that several properties for HPSs on $M$ also hold for HPSs on any Hom-Lie algebroid. Specifically, we do that a HPS $\pi $ on a Hom-Lie algebroid $A$ induce the Hom-Lie algebroid structure on the dual bundle $A^*$ and that a pair of a Hom-Lie algebroid $A$ and the dual bundle $A^*$ has a Hom-Lie bialgebroid structure. Here $A^{*}$ equips with the Hom-Lie algebroid structure induced by a HPS $\pi $. In addition, we investigate HPSs on $M$. In Section \ref{Section:Hom-Poisson-Nijenhuis structures}, we show that there exists a one-to-one correspondence between the pairs consisting of a PS on $M$ and a Poisson isomorphism for it, and HPSs on $M$. This is one of the main theorems in this paper. By this theorem, applications of the Hom-Poisson theory to Poisson geometry can be expected and we see the necessity to study HPSs on arbitrary Hom-Lie algebroid. Then we introduce the notion of a Hom-Nijenhuis structure (HNS) on a Hom-Lie algebroid. This is a generalization of a Nijenhuis structure (NS) on a Lie algebroid. Unlike the ordinary NS on a Lie algebroid, a HNS on Hom-Lie algebroid $(A, \varphi, \phi_{A},[\cdot,\cdot]_{A}, a_{A})$ does not only satisfy the condition that the Nijenhuis torsion vanishes, but also ``the $\phi_{A}$-invariance''. This condition appears in several contexts in this paper and plays an important role for the development of Hom-Lie algebroid theory. Because of $\phi_{A}$-invariance, we can show that HNSs satisfy similar properties to NSs.
 Moreover, we introduce the notion of a Hom-Poisson-Nijenhuis structure (HPNS) on a Hom-Lie algebroid. This is a generalization of a Poisson-Nijenhuis structure (PNS) on a Lie algebroid \cite{MM2},\cite{CNN}. It follows in parallel to PNSs on a Lie algebroid that HPNSs on a Hom-Lie algebroid $A$ equips the hierarchy structure and that they correspond to Hom-Lie bialgebroid structures on $(A,A^{*})$, where $A^{*}$ is the dual bundle of $A$. Their results are also the main theorems in this paper. See \cite{V} and \cite{K3} for PNSs version of the above results respectively.
 
Second, another aim of this paper is to study Hom-Dirac structures (HDS) on a Hom-Courant algebroid. This structure is a generalization of a Dirac structure (DS) on a Courant algebroid \cite{LWX}. While a DS is a maximally isotropic and integrable subbundle on a Courant algebroid, a HDS is defined as a maximally isotropic, integrable and ``the $\phi_{E}$-invariant'' subbundle on a Hom-Courant algebroid $(E, \varphi, \phi_{E},\langle \! \langle \cdot ,\cdot \rangle \! \rangle , \odot_{E},\rho_{E})$. We show that general HDSs on a Hom-Courant algebroid are new Hom-Lie algebroids as in the case of DSs on a Courant algebroid \cite{LWX}. Moreover we show that the graph of a bundle map $H:A^{*}\longrightarrow A$ is a HDS on the Hom-Courant algebroid $A\oplus A^{*}$ constructed by a Hom-Lie bialgebroid $(A,A^{*})$ if and only if $H$ is skew-symmetric and $\phi_{A\oplus A^{*}}$-invariant and $H$ satisfies the Maurer-Cartan type equation:
 \begin{align}
d_{A^{*}}H +\frac{1}{2}[H,H]_{A}=0,
\end{align}
where $d_{A^{*}}$ is the differential of the Hom-Lie algebroid on $A^{*}$ and $[\cdot,\cdot]_{A}$ is the Hom-Schouten bracket on $A$. This result is a generalization of the result for DSs on a Courant algebroid \cite{LWX}. This is also one of the main theorem in this paper. In particular, the graph of any HPS $\pi $ on a Hom-Lie algebroid $A$ is a HDS on the standard Hom-Courant algebroid of $A$.

The paper is organized as follows. In Section \ref{section:Pleriminaries}, we recall several definitions, properties and examples of Hom-Lie algebroid, Hom-Lie bialgebroid and Hom-Courant algebroid. Moreover, we extend definitions of the induced map by $\phi_{A}$  and the Lie derivative on a Hom-Lie algebroid $(A, \varphi, \phi_{A},[\cdot,\cdot]_{A}, a_{A})$ to on $\Gamma(A^{\otimes k}\otimes (A^{*})^{\otimes l})$ and describe their properties. In Section \ref{Section:Hom-Poisson-Nijenhuis structures}, we describe HPSs, HNSs and HPNSs. In subsection \ref{subsection:Hom-Poisson structures}, we begin with the definition of HPSs on arbitrary Hom-Lie algebroid and show that they satisfy similar properties to HPSs on $M$ \cite{CLS}. Moreover, we show a relation between HPSs and PSs on $M$. In subsection \ref{subsection:Hom-Nijenhuis structures} and \ref{sub Hom-Poisson-Nijenhuis structures}, we give the definition and properties of HNSs and HPNSs. The existence of the hierarchy of a HPNS and the correspondence between HPNSs on $A$ and Lie bialgebroid structures on $(A,A^{*})$ are main results in the paper. In Section \ref{Hom-Dirac structures on Hom-Courant algebroids}, we introduce the notion of a HDS on a Hom-Courant algebroid. We show a relation between HDSs and Hom-Lie algebroids. Moreover, we show that  a necessary and sufficient condition for the graph of the bundle map $H:A^{*}\longrightarrow A$ to be a HDS is represented by the Maurer-Cartan type equation, which is one of the main results in the paper.

\section{Preliminaries}\label{section:Pleriminaries}

 \subsection{Hom-Lie algebroids}
 
 In this subsection, we recall the definitions and properties of Hom-Lie algebras, Hom-bundles and Hom-Lie algebroids.
 
 \begin{defn}[\cite{MS},\cite{CLS}]
 A {\it Hom-Lie algebra} is a non-associative algebra $(\mathfrak{g}, [\cdot,\cdot]_{\mathfrak{g}})$ together with an algebra homomorphism $\phi _{\mathfrak{g}}:\mathfrak{g}\longrightarrow \mathfrak{g}$ such that $[\cdot,\cdot]_{\mathfrak{g}}$ is skew-symmetric and the following {\it Hom-Jacobi identity} holds:
 \begin{eqnarray*}
 [\phi_{\mathfrak{g}}(X), [Y,Z]_{\mathfrak{g}}]_{\mathfrak{g}}+[\phi_{\mathfrak{g}}(Y), [Z,X]_{\mathfrak{g}}]_{\mathfrak{g}}+[\phi_{\mathfrak{g}}(Z), [X,Y]_{\mathfrak{g}}]_{\mathfrak{g}}=0
 \end{eqnarray*}
 for any $X,Y$ and $Z$ in $\mathfrak{g}$.
 \end{defn}
 
 Let $M$ be a smooth manifold and $\varphi:M\longrightarrow M$ a smooth map. Then the induced map $\varphi^{*}:C^{\infty}(M)\longrightarrow C^{\infty}(M)$ is given by $ \varphi^{*}f:=f\circ \varphi$
 and is a ring homomorphism, i.e., $\varphi^{*}(fg):=\varphi^{*}f\cdot\varphi^{*}g$
 for any $f$ and $g$ in $C^{\infty}(M)$.
 
 \begin{defn}[\cite{CLS}]
 A {\it Hom-bundle} is a triple $(A,\varphi,\phi_{A})$, where $A\longrightarrow M$ is a vector bundle over $M$, $\varphi :M\longrightarrow M$ is a smooth map and $\phi_{A}: \Gamma(A)\longrightarrow \Gamma (A)$ a linear map, such that $\phi_{A}(fX)=\varphi ^{*}f\cdot \phi_{A}(X)$ 
 for any $f$ in $C^{\infty }(M)$ and $X$ in $\Gamma (A)$. This condition is called {\it $\varphi^{*}$-fuction linear}. A Hom-bundle $(A,\varphi,\phi_{A})$ over $M$ is {\it invertible} if $\varphi$ is a diffeomorphism and $\phi_{A}$ is an invertible linear map.
 \end{defn}
 
  \begin{example}[\cite{CLS}]
Let $B\longrightarrow M$ be a vector bundle over $M$ and $\varphi:M\longrightarrow M$ a smooth map. We denote the pull-back bundle of $B$ along $\varphi$ by $\varphi^{!}B$. For any $X$ in $\Gamma (B)$, we define {\it the pull-back section} $X^{!}$ in $\Gamma (\varphi^{!}B)$ by for any $p$ in $M$, $X^{!}_{p}:=X_{\varphi(p)}$. For any bundle map $\Phi:\varphi^{!}B\longrightarrow B$, we define $\phi_{\varphi^{!}B}:\Gamma (\varphi^{!}B)\longrightarrow \Gamma (\varphi^{!}B)$ by
 \begin{align*}
\phi_{\varphi^{!}B}(X):=(\Phi(X))^{!}
 \end{align*}
 for any $X$ in $\Gamma(\varphi^{!}B)$. Then $(\varphi^{!}B,\varphi,\phi_{\varphi^{!}B})$ is a Hom-bundle.
 \end{example}
 
 \begin{example}[\cite{CLS}]
Let $(A,\varphi,\phi_{A})$ be an invertible Hom-bundle on $M$. We define $\phi_{A}^{\dagger}:\Gamma (A^{*})\longrightarrow \Gamma(A^{*})$ by
 \begin{align*}
 \langle \phi_{A}^{\dagger}(\xi),X\rangle:=\varphi^{*}\langle \xi,\phi_{A}^{-1}(X)\rangle
 \end{align*}
 for any $X$ in $\Gamma(A)$ and $\xi$ in $\Gamma (A^{*})$. Then $(A^{*},\varphi,\phi_{A}^{\dagger})$ is an invertible Hom-bundle. In particular, $(\phi_{A}^{\dagger })^{\dagger }=\phi_{A}$ holds.
 \end{example}
 
 \begin{defn}[\cite{CLS}]\label{Hom-Lie algebroid-def}
 A {\it Hom-Lie algebroid structure} on a vector bundle $A\longrightarrow M$ is a quadruple $(\varphi, \phi_{A},[\cdot,\cdot]_{A}, a_{A})$, where $(A,\varphi,\phi_{A})$ is a Hom-bunble, $(\Gamma (A),$ $\phi_{A},[\cdot,\cdot]_{A})$ is a Hom-Lie algebra and $a_{A}:A\longrightarrow \varphi^{!}TM$ is a bundle map, such that the following conditions are satisfied:
\begin{enumerate}
\item[(i)] For any $X,Y$ in $\Gamma (A)$ and $f$ in $C^{\infty}(M)$, 
$$[X,fY]_{A}=\varphi^{*}f[X,Y]_{A}+a_{A}(\phi _{A}(X))(f)\phi_{A}(Y);$$
\item[(ii)] For any $X$ and $Y$ in $\Gamma (A)$,
            \begin{enumerate}
            \item[(a)] $a_{A}(\phi_{A}(X))\circ \varphi^{*}=\varphi^{*}\circ a_{A}(X)$
            \item[(b)] $a_{A}([X,Y]_{A})\circ \varphi^{*}=a_{A}(\phi_{A}(X))\circ a_{A}(Y)-a_{A}(\phi_{A}(Y))\circ a_{A}(X)$
            \end{enumerate}
\end{enumerate}
 \end{defn}
 
 \begin{rem}
 A Hom-Lie algebroid is just a Lie algebroid if $\varphi=\mathrm{id}_{M}$ and $\phi_{A}=\mathrm{id}_{\Gamma (A)}$.
 \end{rem}
 
 \begin{example}[\cite{CLS}]
 Let $M$ be a smooth manifold and $\varphi :M\longrightarrow M$ a diffeomorphism. 
 Then $\Gamma (\varphi^{!}TM)$ can be identified with the $(\varphi^{*},\varphi^{*})$-derivations on $C^{\infty}(M)$, i.e., for any $f,g$ in $C^{\infty}(M)$ and $X$ in $\Gamma (\varphi^{!}TM)$,
\begin{align}
X(fg)=X(f)\varphi^{*}g+\varphi^{*}f\cdot X(g).
\end{align}
We define $\mathrm{Ad}_{\varphi^{*}}\!:\Gamma (\varphi^{!}TM)\longrightarrow \Gamma (\varphi^{!}TM)$ and $[\cdot,\cdot]_{\varphi^{*}}\!:\Lambda ^{2}\Gamma (\varphi^{!}TM)\longrightarrow \Gamma (\varphi^{!}TM)$ by
\begin{align}
\mathrm{Ad}_{\varphi^{*}}X&:=\varphi^{*}\circ X\circ (\varphi^{*})^{-1}\ \mbox{and}\\
[X,Y]_{\varphi^{*}}&:=\varphi^{*}\circ X\circ (\varphi^{*})^{-1}\circ Y\circ (\varphi^{*})^{-1}-\varphi^{*}\circ Y\circ (\varphi^{*})^{-1}\circ X\circ (\varphi^{*})^{-1}
\end{align}
for any $X$ and $Y$ in $\Gamma (\varphi^{!}TM)$, respectively. Then $(\varphi^{!}TM,\varphi,\mathrm{Ad}_{\varphi^{*}}, [\cdot,\cdot]_{\varphi^{*}},\mathrm{id})$ is a Hom-Lie algebroid on $M$.
 \end{example}

 \begin{rem}\label{def iikae}
 Let $(A,\varphi,\phi_{A})$ be an invertible Hom-bundle. Then the condition (ii) of the definition of a Hom-Lie algebroid can be rewritten as
\begin{enumerate}
\item[(ii)']
            \begin{enumerate}
            \item[(a)] $a_{A}\circ \phi_{A}=\mathrm{Ad}_{\varphi^{*}}\circ a_{A}$
            \item[(b)] $a_{A}([X,Y]_{A})=[a_{A}(X),a_{A}(Y)]_{\varphi^{*}}\ (X,Y\in \Gamma (A)).$
            \end{enumerate}
\end{enumerate}
 \end{rem}
 
 \begin{example}
 For any Hom-bundle $(A,\varphi,\phi_{A})$, we set $[\cdot,\cdot]_{A}:=0$ and $a_{A}:=0$. Then $(A,\varphi,\phi_{A},[\cdot,\cdot]_{A},a_{A})$ is a Hom-Lie algebroid.
 \end{example}
 
 \begin{example}\label{TM oplus R}
The vector bundle $TM\oplus \mathbb{R}:=TM\oplus (M\times \mathbb{R})$ over a manifold $M$ has a natural Lie algebroid structure \cite{MMP}, \cite{S}. As a generalization of this, we introduce a natural Hom-Lie algebroid structure on the vector bundle $\varphi^{!}(TM\oplus \mathbb{R})$ over $M$, where $\varphi:M\longrightarrow M$ is a diffeomorphism. Considering the identification $\Gamma (\varphi^{!}(TM\oplus \mathbb{R}))\cong \Gamma (\varphi^{!}TM)\oplus \varphi^{*}C^{\infty}(M)$, we see that the $C^{\infty }(M)$-module structure on $\Gamma (\varphi^{!}(TM\oplus \mathbb{R}))$ is given by
\begin{align}
f(X,\varphi^{*}g)=(fX,f\varphi^{*}g)
\end{align}
for any $f,g$ in $C^{\infty}(M)$ and $X$ in $\Gamma (\varphi^{!}TM)$. Then we can define the Hom-Lie algebroid structure $(\varphi,\phi,[\cdot,\cdot],a)$ on $\varphi^{!}(TM\oplus \mathbb{R})$ as follows:
\begin{align}
\phi((X,\varphi^{*}f)):&=(\mbox{Ad}_{\varphi^{*}}X,(\varphi^{*})^{2}f);\\
[(X,\varphi^{*}f),(Y,\varphi^{*}g)]:&=([X,Y]_{\varphi^{*}},X(\varphi^{*}g)-Y(\varphi^{*}f))\\
                                    &=([X,Y]_{\varphi^{*}},\varphi^{*}(\mbox{Ad}_{\varphi^{*}}^{-1}(X)(g)-\mbox{Ad}_{\varphi^{*}}^{-1}(Y)(f))); \nonumber\\
a((X,\varphi^{*}f))&=X.
\end{align}
This is a new example of a Hom-Lie algebroid.
 \end{example}

 In the paper, we assume that all underlying Hom-bundles of Hom-Lie algebroids are invertible.
For any Hom-Lie algebroid $(A,\varphi,\phi_A,[\cdot ,\cdot ]_A,a_A)$ on $M$, the map $\phi_{A}$ induces a linear map $\Gamma(A^{\otimes k}\otimes (A^{*})^{\otimes l})\longrightarrow \Gamma(A^{\otimes k}\otimes (A^{*})^{\otimes l})$, which we use the same notaion, by
\begin{align}
\phi_{A}(X_{1}\otimes \dots \otimes &X_{k}\otimes \xi_{1}\otimes \dots \otimes \xi_{l})\nonumber \\
                                    &:=\phi_{A}(X_{1})\otimes \dots \otimes \phi_{A}(X_{k})\otimes \phi_{A}^{\dagger}(\xi_{1})\otimes \dots \otimes \phi_{A}^{\dagger}(\xi_{l})
\end{align}
for any $X_{1},\dots,X_{k}$ in $\Gamma (A)$, $\xi_{1},\dots,\xi_{l}$ in $\Gamma (A^{*})$. In particular, the induced map on $\Gamma ((A^{*})^{\otimes l})$ is denoted by $\phi_{A}^{\dagger}$.

 \subsection{Differential calculus on Hom-Lie algebroids}
 
  In this subsection, we recall the definitions and properties of the differential, the interior multiplication and the Lie derivative on Hom-Lie algebroids.
 
 \begin{defn}[\cite{CLS}]\label{the differential of the Hom-Lie algebroid}
Let $M$ be a manifold, $(A,\varphi,\phi_A,[\cdot ,\cdot ]_A,a_A)$ a Hom-Lie algebroid over $M$. Then an operator $d_A:\Gamma (\Lambda ^k A^*)\longrightarrow \Gamma (\Lambda ^{k+1} A^*)$ is the {\it differential of the Hom-Lie algebroid} $A$ if for any $\omega $ in $\Gamma (\Lambda^k A^*)$ and $X_0,\dots ,X_k$ in $\Gamma(A)$,
\begin{align}\label{A-gaibibun}
(d_A &\omega)(X_0, \dots ,X_k)\nonumber \\
&=\sum_{i=0}^k(-1)^ia_A (X_i)(\omega(\phi_{A}^{-1}(X_0),\dots ,\widehat{\phi_{A}^{-1}(X_i)},\dots ,\phi_{A}^{-1}(X_k)))\nonumber \\
&\quad+\sum _{i<j}^{}(-1)^{i+j}\phi_{A}^{\dagger}(\omega )([\phi_{A}^{-1}(X_i),\phi_{A}^{-1}(X_j)]_A , X_0,\dots ,
\hat{X_i},\dots ,\hat{X_j},\dots ,X_k).
\end{align}
\end{defn}
 
\begin{prop}[\cite{CLS}]
The differential $d_{A}$ of the Hom-Lie algebroid $(A,\varphi,$ $\phi_A,[\cdot ,\cdot ]_A,a_A)$ satisfies the following properties:
\begin{align}
&d_{A}^{2}=0;\\
&d_{A}\circ \phi_{A}^{\dagger}=\phi_{A}^{\dagger}\circ d_{A};\\
&d_{A}(\omega \wedge \eta)=d_{A}\omega \wedge \phi_{A}^{\dagger }(\eta)+(-1)^{k}\phi_{A}^{\dagger }(\omega )\wedge D_{A}\eta 
\end{align}
for any $\omega $ in $\Gamma (\Lambda ^{k}A^{*})$ and $\eta $ in $\Gamma (\Lambda ^{*}A^{*})$.
\end{prop}

 \begin{rem}
{\it A Jacobi algebroid} is a Lie algebroid $A$ together with a $d_{A}$-closed cosection $\rho$ in $\Gamma (A^{*})$ \cite{IW}, \cite{Sh}. Then we can define {\it a Hom-Jacobi algebroid} as a pair of a Hom-Lie algebroid $A$ and a $d_{A}$-closed cosection $\rho$ in $\Gamma (A^{*})$. However a pair $(\varphi^{!}(TM\oplus \mathbb{R}), \phi_{A}((0,1)))$ is not an example of a Hom-Jacobi algebroid while a pair $(TM\oplus \mathbb{R}, (0,1))$ is an example of Jacobi algebroid, where $(0,1)$ is an element in $\Gamma (TM\oplus \mathbb{R})\cong \mathfrak{X}(M)\oplus C^{\infty}(M)$. It is interesting to investigate if there exists a Hom-Jacobi structure on $\varphi^{!}(TM\oplus \mathbb{R})$.
 \end{rem}
 
 Given a Hom-Lie algebroid $(A,\varphi,\phi_{A},[\cdot ,\cdot ]_{A},a_{A})$, the {\it Hom-Schouten bracket} $[\cdot ,\cdot ]_A:\Gamma (\Lambda ^kA)\times \Gamma (\Lambda ^lA)\longrightarrow \Gamma (\Lambda ^{k+l-1}A)$ is defined as the unique extension of the Lie bracket $[\cdot ,\cdot ]_A$ on $\Gamma (A)$ such that
\begin{align*}
&[f,g]_A=0;\\
&[X,f]_A=a_A(\phi_{A}(X))(f);\\
&[X,Y]_A\ \mbox{is the Hom-Lie bracket on}\ \Gamma (A);\\
&[D_1,D_2\wedge D_3]_{A}=[D_1,D_2]_{A}\wedge \phi_{A}(D_3)+(-1)^{\left(k+1\right)l}\phi_{A}(D_2)\wedge [D_1, D_3]_{A};\\
&[D_1,D_2]_A=-(-1)^{(k-1)(l-1)}[D_2,D_1]_A
\end{align*}
for any $f,g$ in $C^\infty(M)$, $X,Y$ in $\Gamma (A)$, $D_1$ in $\Gamma (\Lambda ^kA)$, $D_2$ in $\Gamma (\Lambda ^lA)$ and $D_3$ in $\Gamma (\Lambda ^*A)$.

For any $D$ in $\Gamma (\Lambda ^{k}A)$, we define the {\it interior multiplication} $\iota_{D}^{A}:\Gamma (\Lambda ^{m}A^{*})\longrightarrow \Gamma (\Lambda ^{m-k}A^{*})$by
 \begin{eqnarray*}
(\iota _{D}^{A}\omega )(X_{1},\dots ,X_{m}):=(\phi_{A}^{\dagger}(\omega))(\phi_{A}(D),X_{1},\dots ,X_{m})
 \end{eqnarray*}
 for any $\omega $ in $\Gamma (\Lambda ^{m}A^{*})$ and $X_{i}$ in $\Gamma (A)$. 
 
 For any $X$ in $\Gamma (A)$, the {\it Lie derivative} $\mathcal{L}_X^A$ is defined on $\Gamma (\Lambda ^k A^*)$ by the {\it Hom-Cartan formula}
\begin{eqnarray}\label{A-Lie derivative}
\mathcal{L}_X^A\circ \phi_{A}^{\dagger}=\iota _X^{A}\circ d_A+d_A \iota_{\phi_{A}^{-1}(X)} ^{A} 
\end{eqnarray}
and are extended on $\Gamma (\Lambda ^*A)$ by
\begin{align}\label{A-Lie derivative extension}
\mathcal{L}_X^AD=[X,D]_A 
\end{align}
for any $D$ in $\Gamma (\Lambda ^{*}A)$. Then we have
\begin{align}
\langle \mathcal{L}_X^A\alpha ,Y\rangle =a_{A}(\phi_{A}(X))\langle \alpha ,\phi_{A}^{-1}(Y)\rangle -\langle \phi_{A}^{\dagger }(\alpha ), \mathcal{L}_{X}^{A}(\phi_{A}^{-1}(Y))\rangle 
\end{align}
for any $\alpha $ in $\Gamma (A^{*})$ and $Y$ in $\Gamma (A)$. Moreover, we define the Lie derivative on $\Gamma(A^{\otimes k}\otimes (A^{*})^{\otimes l})$ by
\begin{align}
\mathcal{L}_X^A&(X_{1}\otimes \dots \otimes X_{k}\otimes \xi_{1}\otimes \dots \otimes \xi_{l})\nonumber \\
&:=\sum_{i=1}^{k}\phi_{A}(X_{1})\otimes \dots \otimes \mathcal{L}_X^AX_{i}\otimes \dots \otimes \phi_{A}(X_{k})\otimes \phi_{A}^{\dagger}(\xi_{1})\otimes \dots \otimes \phi_{A}^{\dagger}(\xi_{l})\nonumber \\
&+\sum_{j=1}^{l}\phi_{A}(X_{1})\otimes \dots \otimes \phi_{A}(X_{k})\otimes \phi_{A}^{\dagger}(\xi_{1})\otimes \dots \otimes \mathcal{L}_X^A(\xi_{j})\otimes \dots \otimes \phi_{A}^{\dagger}(\xi_{l})
\end{align}
for any $X_{i}$ in $\Gamma (A)$ and $\xi_{j}$ in $\Gamma (A^{*})$, which is well-defined.

 \subsection{Hom-Lie bialgebroids and Hom-Courant algebroids}
 
  In this subsection, we recall the definitions and properties of Hom-Lie bialgebroids and Hom-Courant algebroids.
 
 \begin{defn}[\cite{CLS}]
 Let $(A, \varphi,\phi_{A})$ be an invertible Hom-bundle, $(A, \varphi,\phi_{A},$ $[\cdot,\cdot]_{A},a_{A})$ and $(A^{*}, \varphi,\phi_{A}^{\dagger},[\cdot,\cdot]_{A^{*}},a_{A^{*}})$ two Hom-Lie algebroids in duality. Then a pair $(A,A^{*})$ is a {\it Hom-Lie bialgebroid} if for any $X$ and $Y$ in $\Gamma (A)$,
 \begin{eqnarray*}
 d_{A^{*}}[X,Y]_{A}=[d_{A^{*}}X,\phi_{A}(Y)]_{A}+[\phi_{A}(X),d_{A^{*}}Y]_{A},
 \end{eqnarray*}
 where $d_{A^{*}}$ is the differential of the Hom-Lie algebroid $A^{*}$.
 \end{defn}
 
 \begin{prop}[\cite{CLS}]\label{Hom-Lie bilagebroid gyaku ok}
 If $(A, A^{*})$ is a Hom-Lie bialgebroid, then so is $(A^{*},A)$.
 \end{prop}
 
  \begin{example}[\cite{CLS}]\label{trivial example}
 For any Hom-Lie algebroid $(A,\varphi,\phi_{A}, [\cdot,\cdot]_{A},a)$ and its dual bundle $A^{*}$ with the trivial Hom-Lie algebroid structure $(\varphi ,\phi_{A}^{\dagger}, 0,0)$, a pair $(A,A^{*})$ is a Hom-Lie bialgebroid.
 \end{example}
 
 \begin{defn}[\cite{CLS}]\label{Hom-Courant algebroid def}
 A {\it Hom-Courant algebroid structure} on a vector bundle $E\longrightarrow M$ is a 5-tuple $(\varphi, \phi_{E},\langle \! \langle \cdot ,\cdot \rangle \! \rangle , \odot_{E},\rho_{E})$, where $(E,\varphi,\phi_{E})$ is an invertible Hom-bunble, $\langle \! \langle \cdot ,\cdot \rangle \! \rangle $ is a non-degenerate symmetric bilinear form on $E$, a map $\odot_{E}:\Gamma (E)\times \Gamma (E)\longrightarrow \Gamma (E)$ is bilinear and $\rho_{E}:E\longrightarrow \varphi^{!}TM$ is a bundle map, such that the following conditions are satisfied:
 \begin{enumerate}
 \item[(i)] $(\Gamma (E), \odot _{E}, \phi_{E})$ is a {\it Hom-Leibniz algebra}, i.e., for any $e_{1},e_{2}$ and $e_{3}$ in $\Gamma (E)$,
            \begin{enumerate}
            \item[(a)] $\phi_{E}(e_{1}\odot _{E}e_{2})=\phi _{E}(e_{1})\odot _{E}\phi_{E}(e_{2})$;
            \item[(b)] $\phi_{E}(e_{1})\odot _{E}(e_{2}\odot _{E}e_{3})=(e_{1}\odot _{E}e_{2})\odot _{E}\phi_{E}(e_{3})+\phi_{E}(e_{2})\odot _{E}(e_{1}\odot _{E}e_{3})$;
            \end{enumerate}
 \item[(ii)] $\rho_{E}\circ \phi_{E}=\mathrm{Ad}_{\varphi^{*}}\circ \rho_{E}$;
 \item[(iii)] For any $e_{1}$ and $e_{2}$ in $\Gamma (E)$, $\rho_{E}(e_{1}\odot _{E}e_{2})=[\rho_{E}(e_{1}),\rho_{E}(e_{2})]_{\varphi^{*}}$;
 \item[(iv)] For any $e$ in $\Gamma (E)$, $e\odot _{E}e=\mathcal{D}\langle \! \langle e,e\rangle \! \rangle $;
 \item[(v)] For any $e_{1}$ and $e_{2}$ in $\Gamma (E)$, $\langle \! \langle  \phi_{E}(e_{1}),\phi_{E}(e_{2})\rangle \! \rangle =\varphi^{*}\langle \! \langle e_{1},e_{2}\rangle \! \rangle $;
 \item[(vi)] For any $e_{1},e_{2}$ and $e$ in $\Gamma (E)$, 
            $$\rho_{E}(\phi_{E}(e))(\langle \! \langle e_{1},e_{2}\rangle \! \rangle )=\langle \! \langle e\odot _{E}e_{1},\phi_{E}(e_{2})\rangle \! \rangle +\langle \! \langle \phi_{E}(e_{1}),e\odot _{E}e_{2}\rangle \! \rangle ,$$
 \end{enumerate}
 where $\mathcal{D}:C^{\infty}(M)\longrightarrow \Gamma (E)$ is difined by
 \begin{eqnarray*}
 \langle \! \langle \mathcal{D}f,e\rangle \! \rangle =\frac{1}{2}\rho_{E}(e)f 
 \end{eqnarray*}
 for any $e$ in $\Gamma (E)$ and $f$ in $C^{\infty}(M)$. We set
 \begin{eqnarray*}
 [\![ e_{1},e_{2}]\!]_{E}:=\frac{1}{2}(e_{1}\odot _{E}e_{2}-e_{2}\odot _{E}e_{1})
 \end{eqnarray*}
 for any $e_{1}$ and $e_{2}$ in $\Gamma (E)$. We call $[\![\cdot,\cdot ]\!]_{E}$ a {\it Hom-Courant bracket} on $\Gamma (E)$.
 \end{defn}
 
 \begin{rem}
In general, a Hom-Courant algebroid is NOT an $E$-Courant algebroid \cite{ChLS}. Hom-Courant algebroids are different generalizations from $E$-Courant algebroids.
 \end{rem}
 
 \begin{prop}[\cite{CLS}]\label{function multiplication for odot}
  Let $(E, \varphi, \phi_{E},\langle \! \langle \cdot ,\cdot \rangle \! \rangle , \odot_{E},\rho_{E})$ be a Hom-Courant algebroid. Then for any $e_{1},e_{2}$ in $\Gamma (E)$ and $f$ in $C^{\infty}(M)$,
 \begin{align*}
 e_{1}\odot_{E}(fe_{2})&=\varphi^{*}f\cdot e_{1}\odot_{E}e_{2}+\rho_{E}(\phi_{E}(e_{1}))(f)\phi_{E}(e_{2}),\\
 (fe_{1})\odot_{E}e_{2}&=\varphi^{*}f\cdot e_{1}\odot_{E}e_{2}-\rho_{E}(\phi_{E}(e_{2}))(f)\phi_{E}(e_{1})+\mathcal{D}f\varphi^{*}\langle \! \langle e_{1},e_{2}\rangle \! \rangle .
\end{align*}
 \end{prop}
 
 For the Hom-Courant bracket $[\![\cdot,\cdot ]\!]_{E}$ on a Hom-Courant algebroid $(E, \varphi, \phi_{E},$ $\langle \! \langle \cdot ,\cdot \rangle \! \rangle , \odot_{E},\rho_{E})$, the following proposition follows immediately from Theorem 5.9 in \cite{CLS}.
  
  \begin{prop}\label{Jacobi-modoki}
  The Hom-Courant bracket $[\![\cdot,\cdot ]\!]_{E}$ on a Hom-Courant algebroid $(E, \varphi, \phi_{E},\langle \! \langle \cdot ,\cdot \rangle \! \rangle , \odot_{E},\rho_{E})$ satisfies that for any $e_{1},e_{2}$ and $e_{3}$ in $\Gamma (E)$,
 \begin{align*}
\sum_{\mathrm{Cycl}(e_{1},e_{2},e_{3})}^{}\!\![\![ [\![ e_{1},e_{2}]\!]_{E}, \phi_{E}(e_{3})]\!]_{E}=\mathcal{D}T(e_{1},e_{2},e_{3}),
\end{align*}
where $T:\Gamma(E)\times\Gamma(E)\times\Gamma(E)\longrightarrow C^{\infty}(M)$ is given by
 \begin{align*}
T(e_{1},e_{2},e_{3}):=\frac{1}{3}\sum_{\mathrm{Cycl}(e_{1},e_{2},e_{3})}^{}\!\!\langle \! \langle [\![ e_{1},e_{2}]\!]_{E}, \phi_{E}(e_{3})\rangle \!\rangle
\end{align*}
for any $e_{1},e_{2}$ and $e_{3}$ in $\Gamma (E)$.
 \end{prop}

The following theorem shows the relationship between a Hom-Lie bialgebroid and a Hom-Courant algebroid.
 
 \begin{theorem}[\cite{CLS}]\label{Hom-Lie bi to Hom-Courant}
 Let $(A,A^{*})$ be a Hom-Lie bialgebroid over $M$. Then $(E=A\oplus A^{*},\varphi, \phi_{E}, \langle \! \langle \cdot ,\cdot \rangle \! \rangle ,\odot_{E}, \rho_{E})$ is a Hom-Courant algebroid, where
 \begin{align*}
\phi_{E}(X+\xi )&:=\phi_{A}(X)+\phi_{A}^{\dagger}(\xi);\\
\langle \! \langle X+\xi,Y+\eta \rangle \! \rangle &:=\frac{1}{2}\left(\langle\xi,Y\rangle+\langle\eta,X\rangle \right);\\
(X+\xi)\odot_{E}(Y+\eta)&:=([X,Y]_{A}+\mathcal{L}_{\xi}^{A^{*}}Y-\iota _{\eta}^{A^{*}}d_{A^{*}}(\phi_{A}^{-1}(X)))\\
                        &\qquad +([\xi,\eta]_{A^{*}}+\mathcal{L}_{X}^{A}\eta-\iota _{Y}^{A}d_{A}((\phi_{A}^{\dagger})^{-1}(\xi)));\\
\rho_{E}(X+\xi)&:=a_{A}(X)+a_{A^{*}}(\xi)
\end{align*}
for any $X,Y$ in $\Gamma (A)$, $\xi $ and $\eta $ in $\Gamma (A^{*})$. By easy calculation, we show that the Hom-Courant bracket on $E$ is given by for any $X,Y$ in $\Gamma (A)$, $\xi$ and $\eta $ in $\Gamma (A^{*})$,
 \begin{align}\label{Hom-Courant bracket for Hom-Lie bi}
[\![ X+\xi,Y+\eta]\!]_{E}&=\left([X,Y]_{A}+\mathcal{L}_{\xi}^{A^{*}}Y-\mathcal{L}_{\eta}^{A^{*}}X+\frac{1}{2}d_{A^{*}}(\langle X,\eta\rangle -\langle Y,\xi\rangle)\right) \nonumber \\
                        &\quad +\left([\xi,\eta]_{A^{*}}+\mathcal{L}_{X}^{A}\eta-\mathcal{L}_{Y}^{A}\xi-\frac{1}{2}d_{A}(\langle X,\eta\rangle -\langle Y,\xi\rangle)\right).
\end{align}
 \end{theorem}
 
In particular, for any Hom-Lie algebroid $A$, the Hom-Courant algebroid $A\oplus A^{*}$ is called {\it the standard Hom-Courant algebroid} of $A$, where $(A,A^{*})$ is a Lie bialgebroid in Example \ref{trivial example}.
 

\section{Hom-Poisson-Nijenhuis structures}\label{Section:Hom-Poisson-Nijenhuis structures}

 \subsection{Hom-Poisson structures}\label{subsection:Hom-Poisson structures}
 
In \cite{CLS}, the authors defined Hom-Poisson structures on a manifold $M$. In this subsection, we introduce Hom-Poisson structures on arbitrary Hom-Lie algebroid and show similar properties to that on $M$. Moreover we show that there exists a one-to-one correspondence between the Hom-Poisson structures on $M$ and the pairs consisting of a Poisson structure on $M$ and a Poisson isomorphism for it.  
 
 \begin{defn}\label{Hom-Poisson structure-def}
 Let $(A, \varphi,\phi_{A},[\cdot,\cdot]_{A},a_{A})$ be a Hom-Lie algebroid and $\pi $ a $2$-section in $\Gamma (\Lambda ^{2}A)$. Then $\pi $ is a {\it Hom-Poisson structure} on $A$ if the following conditions hold: 
 \begin{align}
&[\pi ,\pi ]_{A}=0,\\
&\phi_{A}(\pi )=\pi.\label{phi-invariant-pi}
\end{align}
The condition (\ref{phi-invariant-pi}) is called the {\it $\phi_{A}$-invariance}.
 \end{defn}
 
 In \cite{CLS}, the authors defined a {\it Hom-Poisson structure} $(\pi ,\varphi)$ on a manifold $M$ by the condition that $\pi $ is a Hom-Poisson structure on $\varphi^{!}TM$ in the sense of Definition \ref{Hom-Poisson structure-def}. As well-known, the map
 \begin{align*}
h:&TM\longrightarrow \varphi^{!}TM=\{(p,v)\in M\times TM\ |\ \varphi(p)=\pi_{TM}(v)\},\\
  &v\mapsto (\varphi^{-1}(\pi_{TM}(v)),v)
\end{align*}
is a bundle isomorphism. Therefore a map $(\cdot )^{!}$ is an isomorphism since the induced map $\Gamma (TM)\longrightarrow \Gamma (\varphi^{!}TM)$ by $h$ coincides with $(\cdot )^{!}$. Moreover the map $(\cdot )^{!}$ is extended to the map $\Gamma (\Lambda ^{*}TM)\longrightarrow \Gamma (\Lambda ^{*}\varphi^{!}TM)$. Then we obtain $[X^{!},Y^{!}]_{\varphi^{*}}=\mathrm{Ad}_{\varphi^{*}}[X,Y]^{!}$ for any $X$ and $Y$ in $\mathfrak{X}(M)$. In fact, for any $f$ in $C^{\infty}(M)$ and $p$ in $M$,
 \begin{align*}
([X^{!},Y^{!}]_{\varphi^{*}}(f))(p)&=((\varphi^{*}\circ X^{!}\circ (\varphi^{*})^{-1}\circ Y^{!}\circ (\varphi^{*})^{-1}\\
                                &\qquad-\varphi^{*}\circ Y^{!}\circ (\varphi^{*})^{-1}\circ X^{!}\circ (\varphi^{*})^{-1})(f))(p)\\
                                &=\varphi^{*}(X^{!}((\varphi^{*})^{-1}(Y^{!}((\varphi^{*})^{-1}f))))(p)\\
                                &\qquad-\varphi^{*}(Y^{!}((\varphi^{*})^{-1}(X^{!}((\varphi^{*})^{-1}f))))(p)\\
                                &=X_{\varphi^{2}(p)}(Y((\varphi^{*})^{-1}(f)))-Y_{\varphi^{2}(p)}(X((\varphi^{*})^{-1}(f)))\\
                                &=[X,Y]_{\varphi^{2}(p)}((\varphi^{*})^{-1}(f))\\
                                &=(((\varphi^{*})^{2}\circ [X,Y]\circ (\varphi^{*})^{-1})(f))(p)\\
                                &=(\mathrm{Ad}_{\varphi^{*}}([X,Y]^{!})(f))(p).
\end{align*}
By using properties of the Schouten bracket $[\cdot,\cdot]$ on $\Gamma (\Lambda ^{*}TM)$ and the Hom-Schouten bracket $[\cdot,\cdot]_{\varphi}$ on $\Gamma (\Lambda ^{*}\varphi^{!}TM)$, we have $[D_{1}^{!},D_{2}^{!}]_{\varphi^{*}}=\mathrm{Ad}_{\varphi^{*}}[D_{1},D_{2}]^{!}$ for any $D_{i}$ in $\Gamma (\Lambda ^{*}TM)$.
 Therefore for any $2$-vector field $\pi $ in $\Gamma (\Lambda ^{2}TM)$, 
\begin{align*}
[\pi^{!},\pi^{!}]_{\varphi^{*}}=\mathrm{Ad}_{\varphi^{*}}[\pi,\pi]^{!}.
\end{align*}
Since $\mathrm{Ad}_{\varphi^{*}}$ and $(\cdot)^{!}$ are isomorphisms, $\pi $ is a Poisson structure, i.e., $[\pi,\pi]=0$ if and only if $[\pi^{!},\pi^{!}]_{\varphi^{*}}=0$. We also obtain $(\varphi_{*}X)^{!}=\mathrm{Ad}_{\varphi^{*}}^{-1}(X^{!})$ for any $X$ in $\mathfrak{X}(M)$. In fact, for any $f$ in $C^{\infty}(M)$ and $p$ in $M$,
 \begin{align*}
((\varphi_{*}X)^{!}(f))(p)&=(\varphi_{*}X)_{\varphi(p)}f
                          =(\varphi_{*p}X_{p})f=\langle (df)_{\varphi(p)},\varphi_{*p}X_{p}\rangle\\
                          &=\langle \varphi^{*}_{p}(df)_{p},X_{p}\rangle=\langle d(\varphi^{*}f)_{p},X_{p}\rangle
                          =X_{p}(\varphi^{*}f)\\
                          &=X_{\varphi^{-1}(p)}^{!}(\varphi^{*}f)
                          =((\varphi^{*})^{-1}(X^{!}(\varphi^{*}f)))(p)\\
                          &=((\mathrm{Ad}_{\varphi^{*}}^{-1}(X^{!}))(f))(p).
\end{align*}
Then it follows that $\varphi$ is a Poisson isomorphism for a Poisson structure $\pi $, i.e., $\varphi_{*}\pi=\pi$ if and only if $\pi^{!}$ in $\Gamma (\Lambda ^{2}\varphi^{!}TM)$ is $\mathrm{Ad}_{\varphi^{*}}$-invariant, i.e., $\mathrm{Ad}_{\varphi^{*}}(\pi^{!})=\pi^{!}$. Therefore a Hom-Poisson structure $(\pi,\varphi)$ on $\varphi^{!}TM$ is nothing but a pair of a Poisson structure on $M$ and a Poisson isomorphism. From the above, we can obtain the following theorem.

\begin{theorem}
There exists the following one-to-one correspondence:
\begin{eqnarray*}
\left\{
\begin{aligned}
&\mbox{Pairs consisting of} \\
&\mbox{a Poisson structure on }M \mbox{ and}\\
&\mbox{a Poisson isomorphism for it}
\end{aligned}
\right\}
&\xrightarrow{1\mbox{-}1}&\left\{\mbox{Hom-Poisson structures on }M\right\} \\
(\pi ,\varphi)         &\mapsto&(\pi^{!},\varphi).
\end{eqnarray*}
\end{theorem}
 
From now on, we consider Hom-Poisson structures on arbitrary Hom-Lie algebroid and generalize the results for Hom-Poisson structures on $M$ in \cite{CLS}.
 
 \begin{prop}\label{pi invariance and commutativity}
 Let $(A, \varphi,\phi_{A},[\cdot,\cdot]_{A},a_{A})$ be a Hom-Lie algebroid and $\pi $ in $\Gamma (\Lambda ^{2}A)$. Then the following conditions are equivalent:
 \begin{enumerate}
 \item[(i)]  $\phi_{A}(\pi)=\pi$\ (the $\phi_{A}$-invariance);
 \item[(ii)] $\phi_{A}\circ \pi^{\sharp }=\pi ^{\sharp} \circ \phi_{A}^{\dagger}$.
 \end{enumerate}
 \end{prop}
 
 \begin{proof}
 We set $\pi :=\sum_{i}X_{i}\wedge Y_{i}$, where $X_{i}$ and $Y_{i}$ are in $\Gamma (A)$. We have $\phi_{A}(\pi)=\sum_{i}\phi_{A}(X_{i})\wedge \phi_{A}(Y_{i})$. Then we obtain $\phi_{A}\circ \pi^{\sharp}=\phi_{A}(\pi )^{\sharp}\circ \phi_{A}^{\dagger}$. In fact, for any $\alpha $ in $\Gamma (A^{*})$,
 \begin{align*}
\phi_{A}(\pi^{\sharp}\alpha )&=\phi_{A}\left(\sum_{i}(\langle \alpha ,X_{i}\rangle Y_{i}-\langle \alpha ,Y_{i}\rangle X_{i})\right)\\
                             &=\sum_{i}(\varphi^{*}\langle \alpha ,X_{i}\rangle \phi_{A}(Y_{i})-\varphi^{*}\langle \alpha ,Y_{i}\rangle \phi_{A}(X_{i}))\\
                             &=\sum_{i}(\varphi^{*}\langle \alpha ,\phi_{A}^{-1}(\phi_{A}(X_{i}))\rangle \phi_{A}(Y_{i})-\varphi^{*}\langle \alpha ,\phi_{A}^{-1}(\phi_{A}(Y_{i}))\rangle \phi_{A}(X_{i}))\\
                             &=\sum_{i}(\langle \phi_{A}^{\dagger}(\alpha ),\phi_{A}(X_{i})\rangle \phi_{A}(Y_{i})-\langle \phi_{A}^{\dagger}(\alpha ),\phi_{A}(Y_{i})\rangle \phi_{A}(X_{i}))\\
                             &=\phi_{A}(\pi )^{\sharp}(\phi_{A}^{\dagger}(\alpha )).
\end{align*} 
This equality means the equivalency between (i) and (ii).
 \end{proof}
 
 The following two theorems can be proved as in the case of $\varphi^{!}T_{\pi}^{*}M$ in \cite{CLS}.
 
 \begin{theorem}
 Let $\pi $ be a Hom-Poisson structure on a Hom-Lie algebroid $(A, \varphi,\phi_{A},[\cdot,\cdot]_{A},a_{A})$. Then $A^{*}_{\pi}=(A^{*},\varphi, \phi_{A}^{\dagger}, [\cdot,\cdot]_{\pi} ,\pi ^{\sharp})$ is a Hom-Lie algebroid, where
 \begin{align}\label{pi bracket}
[\xi,\eta]_{\pi}:=\mathcal{L}_{\pi^{\sharp}\xi}^{A}\eta-\mathcal{L}_{\pi^{\sharp}\eta}^{A}\xi-d_{A}\langle \pi^{\sharp}\xi,\eta\rangle
\end{align}
for any $\xi $ and $\eta $ in $\Gamma (A^{*})$.
 \end{theorem}
 
 \begin{theorem}
 Let $\pi $ be a Hom-Poisson structure on a Hom-Lie algebroid $(A, \varphi,\phi_{A},[\cdot,\cdot]_{A},a_{A})$. If the set $\{d_{\pi}f\ |\ f\in C^{\infty }(M)\}$ generates $\Gamma (A^{*})$ as a $C^{\infty }(M)$-module, then a pair $(A,A^{*}_{\pi})$ is a Hom-Lie bialgebroid, where $d_{\pi}$ is the differential of the Hom-Lie algebroid $A_{\pi}^{*}$.
 \end{theorem}
 
We also denote the operator defined by the formula (\ref{A-gaibibun}) for any $2$-section $\pi $ in $\Gamma (\Lambda^{2}A)$ by the same symbol $d_{\pi}$. The following two propositions will be used in Subsection \ref{sub Hom-Poisson-Nijenhuis structures} and Section \ref{Hom-Dirac structures on Hom-Courant algebroids}. These are generalizations of properties for Poisson structures \cite{KMa}.

%
 
 \begin{prop}\label{d pi property}
 Let $(A, \varphi,\phi_{A},[\cdot,\cdot]_{A},a_{A})$ be a Hom-Lie algebroid and $\pi $ in $\Gamma (\Lambda ^{2}A)$ a $\phi_{A}$-invariant $2$-section. Then for any $D$ in $\Gamma(\Lambda ^{*}A)$,
 \begin{align}\label{pi bracket property}
 d_{\pi}D=[\pi ,D]_{A}.
\end{align}
 \end{prop}
 
 \begin{proof}
 It is sufficient to prove that $d_{\pi}Z=[\pi,Z]_{A}$ holds for any $Z$ in $\Gamma (A)$. We set $\pi =\sum_{i}X_{i}\wedge Y_{i}\ (X_{i}, Y_{i}\in \Gamma (A))$.  By long calculations, we have that both $(d_{\pi}Z)(\alpha ,\beta)$ and $[\pi,Z]_{A}(\alpha ,\beta)$ for any $\alpha $ and $\beta$ in $\Gamma (A^{*})$ are equal to
 \begin{align*}
\sum_{i}(&\langle X_{i},\alpha \rangle a_{A}(Y_{i})\langle Z, (\phi_{A}^{\dagger })^{-1}(\beta)\rangle -\langle Y_{i},\alpha \rangle a_{A}(X_{i})\langle Z, (\phi_{A}^{\dagger })^{-1}(\beta)\rangle \\
         &-\langle X_{i},\beta \rangle a_{A}(Y_{i})\langle Z, (\phi_{A}^{\dagger })^{-1}(\alpha )\rangle +\langle Y_{i},\beta \rangle a_{A}(X_{i})\langle Z, (\phi_{A}^{\dagger })^{-1}(\alpha )\rangle \\
         &-\langle \phi_{A}(X_{i}),\alpha \rangle \langle \phi_{A}(Z),\mathcal{L}_{Y_{i}}^{A}((\phi_{A}^{\dagger })^{-1}(\beta))\rangle \\
         &+\langle \phi_{A}(Y_{i}),\alpha \rangle \langle \phi_{A}(Z),\mathcal{L}_{X_{i}}^{A}((\phi_{A}^{\dagger })^{-1}(\beta))\rangle \\
         &+\langle \phi_{A}(X_{i}),\beta \rangle \langle \phi_{A}(Z),\mathcal{L}_{Y_{i}}^{A}((\phi_{A}^{\dagger })^{-1}(\alpha))\rangle \\
         &-\langle \phi_{A}(Y_{i}),\beta \rangle \langle \phi_{A}(Z),\mathcal{L}_{X_{i}}^{A}((\phi_{A}^{\dagger })^{-1}(\alpha))\rangle .
\end{align*} 
Therefore (\ref{pi bracket property}) holds.
 \end{proof}
 
 \begin{prop}\label{pi pi calculus}
 Let $(A, \varphi,\phi_{A},[\cdot,\cdot]_{A},a_{A})$ be a Hom-Lie algebroid and $\pi $ in $\Gamma (\Lambda ^{2}A)$ a $\phi_{A}$-invariant $2$-section. Then for any $\alpha $ and $\beta $ in $\Gamma(A^{*})$,
 \begin{align}\label{pi pi culculus formula}
 \frac{1}{2}[\pi,\pi]_{A}(\phi_{A}^{\dagger}(\alpha ),\phi_{A}^{\dagger}(\beta),\cdot)=[\pi ^{\sharp}\alpha ,\pi^{\sharp}\beta ]_{A}-\pi^{\sharp}[\alpha ,\beta ]_{\pi}.
\end{align}
 \end{prop}
 
 \begin{proof}
By Proposition \ref{d pi property}, we have $[\pi,\pi]_{A}=d_{\pi}\pi$ since $\pi$ is $\phi_{A}$-invariant. By calculations using the formula (\ref{A-gaibibun}), the $\phi_{A}$-invariance of $\pi$ and the definition (\ref{pi bracket}) of $[\cdot,\cdot]_{\pi}$, we get, for any $\alpha ,\beta $ and $\gamma $ in $\Gamma (A^{*})$,
 \begin{align*}
[\pi,\pi](\alpha ,\beta ,\gamma )=2\langle \gamma , [\pi^{\sharp}((\phi_{A}^{\dagger})^{-1}(\alpha )),& \pi^{\sharp}((\phi_{A}^{\dagger})^{-1}(\beta))]_{A}\\
&-\pi^{\sharp}[(\phi_{A}^{\dagger})^{-1}(\alpha ),(\phi_{A}^{\dagger})^{-1}(\beta )]_{\pi}\rangle .
\end{align*}
By replacing $\alpha ,\beta$ with $\phi_{A}^{\dagger}(\beta ), \phi_{A}^{\dagger}(\beta )$ respectively, we obtain the formula (\ref{pi pi culculus formula}).
 \end{proof}
 
 \begin{rem}
 Considering ``Hom-versions'' of several generalizations of Poisson structures, for example, quasi-Poisson structures \cite{SW}, twisted Poisson structures \cite{AK}, \cite{AKMe}, and so on, is very interesting.
 \end{rem}
 
 \subsection{Hom-Nijenhuis structures}\label{subsection:Hom-Nijenhuis structures}
 
 In this subsection, we introduce Hom-Nijenhuis structures on a Hom-Lie algebroid and investigate properties of these. The condition ``$\phi_{A}$-invariance'' also plays an important role for Hom-Nijenhuis structures.
 
 \begin{defn}
 Let $(A, \varphi,\phi_{A},[\cdot,\cdot]_{A},a_{A})$ be a Hom-Lie algebroid over $M$ and $N:A\longrightarrow A$ a bundle map over $M$. Then we define the {\it Nijenhuis torsion} $\mathcal{T}_{N}$ of $N$ by
 \begin{align}
\mathcal{T}_{N}(X,Y):=[NX,NY]_{A}-N[NX,Y]_{A}-N[X,NY]_{A}+N^{2}[X,Y]_{A}
\end{align}
for any $X$ and $Y$ in $\Gamma (A)$. $N$ is a {\it Hom-Nijenhuis structure} on $A$ if 
 \begin{align}
&\mathcal{T}_{N}=0,\\
&\phi_{A}(N)=N. \label{phi-invariant N}
\end{align}
The condition (\ref{phi-invariant N}) is called the {\it $\phi_{A}$-invariance}.
 \end{defn}
 
The following lemma is fundamental.

 \begin{lemma}\label{phi-N-lemma}
 Let $(A, \varphi,\phi_{A},[\cdot,\cdot]_{A},a_{A})$ be a Hom-Lie algebroid over $M$ and $N,N':A\longrightarrow A$ two bundle maps over $M$. Then the followings hold:
 \begin{enumerate}
 \item[(i)] $\phi_{A}(NX)=\phi_{A}(N)\phi_{A}(X)$ for any $X$ in $\Gamma (A)$;
 \item[(ii)] $\phi_{A}^{\dagger}(N^{*}\xi)=\phi_{A}(N^{*})\phi_{A}^{\dagger}(\xi)$ for any $\xi$ in $\Gamma (A^{*})$;
 \item[(iii)] $\phi_{A}(N^{*})=\phi_{A}(N)^{*}$;
 \item[(iv)] $\phi_{A}(N\circ N')=\phi_{A}(N)\circ \phi_{A}(N')$;
  \item[(v)] $N\circ \phi_{A}=\phi_{A}\circ N \Longleftrightarrow \phi_{A}(N)=N$\ (the $\phi_{A}$-invariance).
 \end{enumerate}
 \end{lemma}
 
\begin{proof}
 We set $N:=\sum_{i}^{}\xi^{i}\otimes X_{i}$, where $\xi^{i}$'s are in $\Gamma (A^{*})$ and $X_{i}$'s are in $\Gamma (A)$. (i) We calculate
\begin{align*}
\phi_{A}(NX)=\phi_{A}\left(\sum_{i}^{}\langle \xi^{i},X\rangle X_{i}\right)=\sum_{i}^{}\varphi^{*}\langle \xi^{i},X\rangle \phi_{A}(X_{i}).
\end{align*} 
On the other hand, by $\phi_{A}(N)=\sum_{i}^{}\phi_{A}^{\dagger }(\xi^{i})\otimes \phi_{A}(X_{i})$, we obtain
\begin{align*}
\phi_{A}(N)\phi_{A}(X)=\sum_{i}^{}\langle \phi_{A}^{\dagger }(\xi^{i}),\phi_{A}(X)\rangle \phi_{A}(X_{i})=\sum_{i}^{}\varphi^{*}\langle \xi^{i},X\rangle \phi_{A}(X_{i}).
\end{align*}
(ii) By $\phi_{A}(N^{*})=\sum_{i}^{}\phi_{A}(X^{i})\otimes \phi_{A}^{\dagger}(\xi_{i})$ since $N^{*}=\sum_{i}^{}X^{i}\otimes \xi_{i}$, we obtain
\begin{align*}
\phi_{A}(N^{*}\xi)&=\phi_{A}\left(\sum_{i}^{}\langle \xi,X_{i}\rangle \xi^{i}\right)=\sum_{i}^{}\varphi^{*}\langle \xi,X_{i}\rangle \phi_{A}^{\dagger}(\xi^{i}),\\
\phi_{A}(N^{*}\phi_{A}^{\dagger }(\xi))&=\sum_{i}^{}\langle \phi_{A}^{\dagger }(\xi),\phi_{A}(X_{i})\rangle \phi_{A}^{\dagger }(\xi^{i})=\sum_{i}^{}\varphi^{*}\langle \xi,X_{i}\rangle \phi_{A}^{\dagger}(\xi^{i}).
\end{align*}
(iii) For any $\alpha $ in $\Gamma (A^{*})$ and $X$ in $\Gamma (M)$, we obtain
\begin{align*}
\langle \phi_{A}(N^{*})\alpha ,X\rangle &=\left\langle \sum_{i}^{}\langle \alpha ,\phi_{A}(X_{i}) \rangle \phi_{A}^{\dagger }(\xi^{i}),X\right\rangle=\sum_{i}^{}\langle \alpha ,\phi_{A}(X_{i}) \rangle \langle \phi_{A}^{\dagger }(\xi^{i}),X\rangle,\\
\langle \phi_{A}(N)^{*}\alpha ,X\rangle &=\langle \alpha ,\phi_{A}(N)X\rangle \\ &=\left\langle  \alpha ,\sum_{i}^{}\langle \phi_{A}^{\dagger }(\xi^{i}), X \rangle \phi_{A}(X_{i})\right\rangle=\sum_{i}^{}\langle \phi_{A}^{\dagger }(\xi^{i}),X\rangle\langle \alpha ,\phi_{A}(X_{i}) \rangle .
\end{align*}
(iv) For any $X$ in $\Gamma (A)$, by using (i), we obtain
\begin{align*}
\phi_{A}(N\circ N')X&=\phi_{A}((N\circ N')(\phi_{A}^{-1}(X)))=\phi_{A}(N(N'(\phi_{A}^{-1}(X))))\\
                    &=\phi_{A}(N)\phi_{A}(N'(\phi_{A}^{-1}(X)))=\phi_{A}(N)(\phi_{A}(N')X)\\
                    &=(\phi_{A}(N)\circ \phi_{A}(N'))X.
\end{align*}
Finally, for any $X$ in $\Gamma (A)$, by using (i), we obtain
\begin{align*}
(N\circ \phi_{A}-\phi_{A}\circ N)(X)&=N(\phi_{A}(X))-\phi_{A}(NX)\\
                                    &=N(\phi_{A}(X))-\phi_{A}(N)(\phi_{A}(X))\\
                                    &=(N-\phi_{A}(N))(\phi_{A}(X)).
\end{align*}
Since $X$ is arbitrary, $\phi_{A}(X)$ is also arbitrary. Therefore (v) holds.
 \end{proof}
 
 Let $(A, \varphi,\phi_{A},[\cdot,\cdot]_{A},a_{A})$ be a Hom-Lie algebroid over $M$. We define the bracket $[\cdot,\cdot]_{N}$ on $\Gamma(A)$ for a bundle map $N:A\longrightarrow A$ over $M$ by
 \begin{align}
&[X,Y]_{N}:=[NX,Y]_{A}+[X,NY]_{A}-N[X,Y]_{A}
\end{align}
for any $X$ and $Y$ in $\Gamma (A)$. Hom-Nijenhuis structures modify Hom-Lie algebroid structures on $A$ as well as Nijenhuis structures do Lie algebroid structures \cite{CNN}, \cite{KMa}, \cite{V}.

 \begin{theorem}
 Let $N$ be a Hom-Nijenhuis structure on a Hom-Lie algebroid $(A, \varphi,\phi_{A},[\cdot,\cdot]_{A},a_{A})$. Then $A_{N}=(A,\varphi, \phi_{A}, [\cdot,\cdot]_{N} ,a_{A}\circ N)$ is a Hom-Lie algebroid.
 \end{theorem}
 
 \begin{proof}
First, we show that $(\Gamma (A),[\cdot,\cdot]_{N},\phi_{A})$ is a Hom-Lie algebra. For any $X$ and $Y$ in $\Gamma (A)$, we culculate
 \begin{align*}
\phi_{A}([X,Y]_{N})&=\phi_{A}([NX,Y]_{A}+[X,NY]_{A}-N[X,Y]_{A})\\
                   &=[\phi_{A}(NX),\phi_{A}(Y)]_{A}+[\phi_{A}(X),\phi_{A}(NY)]_{A}-N(\phi_{A}([X,Y]_{A}))\\
                   &=[N(\phi_{A}(X)),\phi_{A}(Y)]_{A}+[\phi_{A}(X),N(\phi_{A}(Y))]_{A}-N[\phi_{A}(X),\phi_{A}(Y)]_{A}\\
                   &=[\phi_{A}(X),\phi_{A}(Y)]_{N}
\end{align*}
by Lemma \ref{phi-N-lemma} and the assumption that $\phi_{A}$ is an algebra homomorphism with respect to $(\Gamma (A),[\cdot,\cdot]_{A})$. Hence $\phi_{A}$ is an algebra homomorphism with respect to $(\Gamma (A),[\cdot,\cdot]_{N})$. By straightforward computation, it follows that for any $X,Y$ and $Z$ in $\Gamma (A)$,
 \begin{align*}
&\sum_{\mbox{Cycl}(X,Y,Z)}[\phi_{A}(X),[Y,Z]_{N}]_{N}\\
&=\!\!\sum_{\mbox{Cycl}(NX,NY,Z)}\! \![\phi_{A}(NX),[NY,Z]_{A}]_{A}+\!\!\sum_{\mbox{Cycl}(NX,Y,NZ)}\!\![\phi_{A}(NX),[Y,NZ]_{A}]_{A}\\
&\quad +\!\!\sum_{\mbox{Cycl}(X,NY,NZ)}\!\![\phi_{A}(X),[NY,NZ]_{A}]_{A}-N\left(\sum_{\mbox{Cycl}(NX,Y,Z)}\!\![\phi_{A}(NX),[Y,Z]_{A}]_{A}\right)\\
&\quad -N\left(\sum_{\mbox{Cycl}(X,NY,Z)}\!\![\phi_{A}(X),[NY,Z]_{A}]_{A}\right)-N\left(\sum_{\mbox{Cycl}(X,Y,NZ)}\!\![\phi_{A}(X),[Y,NZ]_{A}]_{A}\right)\\
&\quad +N^{2}\left(\sum_{\mbox{Cycl}(X,Y,Z)}[\phi_{A}(X),[Y,Z]_{A}]_{A}\right).
\end{align*}
Since $[\cdot,\cdot]_{A}$ satisfies the Hom-Jacobi identity, the right hand side is equal to $0$. Therefore $(\Gamma (A),[\cdot,\cdot]_{N},\phi_{A})$ is a Hom-Lie algebra.

Second, we compute for any $X,Y$ in $\Gamma (A)$ and $f$ in $C^{\infty}(M)$,
 \begin{align*}
[X,fY]_{N}&=[NX,fY]_{A}+[X,N(fY)]_{A}-N[X,fY]_{A}\\
          &=\varphi^{*}f[NX,Y]_{A}+a_{A}(\phi_{A}(NX))(f)\phi_{A}(Y)\\
          &\quad +\varphi^{*}f[X,NY]_{A}+a_{A}(\phi_{A}(X))(f)\phi_{A}(NY)\\
          &\quad -N(\varphi^{*}f[X,Y]_{A}+a_{A}(\phi_{A}(X))(f)\phi_{A}(Y))\\
          &=\varphi^{*}f([NX,Y]_{A}+[X,NY]_{A}-N[X,Y]_{A})\\
          &\quad +(a_{A}\circ N)(\phi_{A}(X))(f)\phi_{A}(Y),
\end{align*}
where we use Lemma \ref{phi-N-lemma} and the assumption that $(A, \varphi,\phi_{A},[\cdot,\cdot]_{A},a_{A})$ is a Hom-Lie algebroid.

Finally, for any $X$ and $Y$ in $\Gamma (A)$, we obtain
 \begin{align*}
(a_{A}\circ N)(\phi_{A}(X))&=a_{A}(N(\phi_{A}(X)))=a_{A}(\phi_{A}(NX))=\mbox{Ad}_{\varphi^{*}}(a_{A}(NX))\\
                           &=\mbox{Ad}_{\varphi^{*}}((a_{A}\circ N)(X))
\end{align*}
and
 \begin{align*}
(a_{A}\circ N)([X,Y]_{N})&=a_{A}(N[X,Y]_{N})=a_{A}([NX,NY]_{A})\\
                         &=[a_{A}(NX),a_{A}(NY)]_{\varphi^{*}}=[(a_{A}\circ N)(X),(a_{A}\circ N)(Y)]_{\varphi^{*}}
\end{align*}
since $(A, \varphi,\phi_{A},[\cdot,\cdot]_{A},a_{A})$ is a Hom-Lie algebroid and $N$ is Hom-Nijenhuis.

From the above, we conclude that $A_{N}=(A,\varphi, \phi_{A}, [\cdot,\cdot]_{N} ,a_{A}\circ N)$ is a Hom-Lie algebroid. 
 \end{proof}

The Lie derivative on the $(1,1)$-tensor fields satisfies the following.
  
 \begin{prop}
 Let $(A, \varphi,\phi_{A},[\cdot,\cdot]_{A},a_{A})$ be a Hom-Lie algebroid over $M$ and $N,N':A\longrightarrow A$ two bundle maps over $M$. Then the followings hold:
 \begin{enumerate}
 \item[(i)] $(\mathcal{L}_{X}^{A}N)Y=\mathcal{L}_{X}^{A}(N(\phi_{A}^{-1}(Y)))-\phi_{A}(N)\mathcal{L}_{X}^{A}(\phi_{A}^{-1}(Y))$ for any $X$ and $Y$ in $\Gamma (A)$;
 \item[(ii)] $\mathcal{L}_{X}^{A}(N\circ N')=\mathcal{L}_{X}^{A}N\circ \phi_{A}(N')+\phi_{A}(N)\circ \mathcal{L}_{X}^{A}N'$ for any $X$ in $\Gamma (A)$.
 \end{enumerate}
 \end{prop}
 
\begin{proof}
 We set $N:=\sum_{i}^{}\xi^{i}\otimes X_{i}$, where $\xi^{i}$ is in $\Gamma (A^{*})$ and $X_{i}$ is in $\Gamma (A)$. We obtain
\begin{align*}
(\mathcal{L}_{X}^{A}N)Y&=\left(\sum_{i}^{}\mathcal{L}_{X}^{A}\xi^{i}\otimes \phi_{A}(X_{i})+\sum_{i}^{}\phi_{A}^{\dagger }(\xi^{i})\otimes \mathcal{L}_{X}^{A}X_{i}\right)Y\\
                       &=\sum_{i}^{}\langle \mathcal{L}_{X}^{A}\xi^{i}, Y\rangle \phi_{A}(X_{i})+\sum_{i}^{}\langle \phi_{A}^{\dagger }(\xi^{i}),Y\rangle \mathcal{L}_{X}^{A}X_{i}\\
                       &=\sum_{i}^{}(a_{A}(\phi_{A}(X))(\langle \xi^{i}, \phi_{A}^{-1}(Y)\rangle )\\
                       &\qquad \qquad-\langle \phi_{A}^{\dagger }(\xi^{i}), \mathcal{L}_{X}^{A}(\phi_{A}^{-1}(Y))\rangle )\phi_{A}(X_{i})\\
                       &\qquad +\sum_{i}^{}\varphi^{*}\langle \phi_{A}^{\dagger }(\xi^{i}),Y\rangle \mathcal{L}_{X}^{A}X_{i}\\
\mathcal{L}_{X}^{A}(N(\phi_{A}^{-1}(Y)))&=\mathcal{L}_{X}^{A}\left(\sum_{i}\langle \xi^{i}, \phi_{A}^{-1}(Y)\rangle X_{i}\right)\\
                                        &=\sum_{i}(\varphi^{*}\langle \xi^{i}, \phi_{A}^{-1}(Y)\rangle \mathcal{L}_{X}^{A}X_{i}\\
                                        &\qquad \qquad +a_{A}(\phi_{A}(X))(\langle \xi^{i}, \phi_{A}^{-1}(Y)\rangle)\phi_{A}(X_{i}))\\ 
\phi_{A}(N)\mathcal{L}_{X}^{A}(\phi_{A}^{-1}(Y))&=\sum_{i}^{}\langle \phi_{A}^{\dagger }(\xi^{i}), \mathcal{L}_{X}^{A}(\phi_{A}^{-1}(Y))\rangle )\phi_{A}(X_{i}),
\end{align*}
so that (i) holds. For any $Y$ in $\Gamma (A)$, we calculate
\begin{align*}
(\mathcal{L}_{X}^{A}(N\circ N'))Y&=\mathcal{L}_{X}^{A}((N\circ N')(\phi_{A}^{-1}(Y)))-\phi_{A}(N\circ N')\mathcal{L}_{X}^{A}(\phi_{A}^{-1}(Y))\\
                                 &=\mathcal{L}_{X}^{A}(N(N'(\phi_{A}^{-1}(Y))))-\phi_{A}(N)(\phi_{A}(N')\mathcal{L}_{X}^{A}(\phi_{A}^{-1}(Y)))\\
                                 &=(\mathcal{L}_{X}^{A}N)(\phi_{A}(N'(\phi_{A}^{-1}(Y))))+\phi_{A}(N)\mathcal{L}_{X}^{A}(N'(\phi_{A}^{-1}(Y)))\\
                                 &\qquad -\phi_{A}(N)(\phi_{A}(N')\mathcal{L}_{X}^{A}(\phi_{A}^{-1}(Y)))\\
                                 &=(\mathcal{L}_{X}^{A}N)(\phi_{A}(N')Y)\\
                                 &\qquad +\phi_{A}(N)(\mathcal{L}_{X}^{A}(N'(\phi_{A}^{-1}(Y)))-\phi_{A}(N')\mathcal{L}_{X}^{A}(\phi_{A}^{-1}(Y)))\\
                                 &=(\mathcal{L}_{X}^{A}N)(\phi_{A}(N')Y)+\phi_{A}(N)((\mathcal{L}_{X}^{A}N')Y)\\
                                 &=(\mathcal{L}_{X}^{A}N\circ \phi_{A}(N')+\phi_{A}(N)\circ \mathcal{L}_{X}^{A}N')Y.
\end{align*}
Hence (ii) holds.
 \end{proof}
 
 As for Nijenhuis structures \cite{K3}, the following holds for Hom-Nijenhuis structures.
 
 \begin{prop}
 Let $N$ be a Hom-Nijenhuis structure on a Hom-Lie algebroid $(A, \varphi,\phi_{A},[\cdot,\cdot]_{A},a_{A})$ on $M$. Then for any $f$ in $C^\infty(M)$,
 \begin{enumerate}
 \item[(i)]  $d_{N}f=N^{*}d_{A}f$;
 \item[(ii)] $d_{N}d_{A}f=-d_{A}d_{N}f$,
 \end{enumerate}
 where $d_{N}$ and $d_{A}$ are the differentials of the Hom-Lie algebroids $A_{N}$ and $A$ respectively.
 \end{prop}

 \begin{proof}
 (i) For any $X$ in $\Gamma (A)$,
 \begin{align*}
\langle d_{N}f,X\rangle=(a_{A}\circ N)(X)(f)=a_{A}(NX)(f)=\langle d_{A}f,NX\rangle=\langle N^{*}d_{A}f,X\rangle .
\end{align*}
(ii) For any $X$ and $Y$ in $\Gamma (A)$,
 \begin{align*}
(d_{N}d_{A}f)(X,Y)&=(a_{A}\circ N)(X)(\langle d_{A}f, \phi_{A}^{-1}(Y)\rangle)-(a_{A}\circ N)(Y)(\langle d_{A}f, \phi_{A}^{-1}(X)\rangle)\\
                  &\qquad -\langle \phi_{A}^{\dagger}(d_{A}f), [\phi_{A}^{-1}(X),\phi_{A}^{-1} (Y)]_{N}\rangle\\
                  &=(a_{A}(NX)(a_{A}(\phi_{A}^{-1}(Y))(f))-(a_{A}(NY)(a_{A}(\phi_{A}^{-1}(X))(f))\\
                  &\qquad -\varphi^{*}(a_{A}(\phi_{A}^{-1}([\phi_{A}^{-1}(X),\phi_{A}^{-1} (Y)]_{N}))(f))\\
                  &=(a_{A}(NX)(a_{A}(\phi_{A}^{-1}(Y))(f))-(a_{A}(NY)(a_{A}(\phi_{A}^{-1}(X))(f))\\
                  &\qquad -\varphi^{*}(a_{A}(\phi_{A}^{-1}([N(\phi_{A}^{-1}(X)),\phi_{A}^{-1} (Y)]_{A}))(f))\\
                  &\qquad -\varphi^{*}(a_{A}(\phi_{A}^{-1}([\phi_{A}^{-1}(X),N(\phi_{A}^{-1} (Y))]_{A}))(f))\\
                  &\qquad +\varphi^{*}(a_{A}(\phi_{A}^{-1}(N([\phi_{A}^{-1}(X),\phi_{A}^{-1} (Y)]_{A})))(f))\\
(d_{A}d_{N}f)(X,Y)&=a_{A}(X)(\langle d_{N}f, \phi_{A}^{-1}(Y)\rangle)-a_{A}(Y)(\langle d_{N}f, \phi_{A}^{-1}(X)\rangle)\\
                  &\qquad -\langle \phi_{A}^{\dagger}(d_{N}f), [\phi_{A}^{-1}(X),\phi_{A}^{-1} (Y)]_{A}\rangle\\
                  &=a_{A}(X)(a_{A}(N(\phi_{A}^{-1}(Y)))(f))-a_{A}(Y)(a_{A}(N(\phi_{A}^{-1}(X)))(f))\\
                  &\qquad -\varphi^{*}(a_{A}(N(\phi_{A}^{-1}([\phi_{A}^{-1}(X),\phi_{A}^{-1} (Y)]_{A})))(f)),
\end{align*}
so that we calculate
 \begin{align*}
(d_{N}d_{A}f&+d_{A}d_{N}f)(X,Y)\\
            &=(a_{A}(NX)(a_{A}(\phi_{A}^{-1}(Y))(f))-(a_{A}(NY)(a_{A}(\phi_{A}^{-1}(X))(f))\\
            &\qquad -\varphi^{*}(a_{A}(\phi_{A}^{-1}([N(\phi_{A}^{-1}(X)),\phi_{A}^{-1} (Y)]_{A}))(f))\\
            &\qquad -\varphi^{*}(a_{A}(\phi_{A}^{-1}([\phi_{A}^{-1}(X),N(\phi_{A}^{-1} (Y))]_{A}))(f))\\
            &\qquad +\varphi^{*}(a_{A}(\phi_{A}^{-1}(N([\phi_{A}^{-1}(X),\phi_{A}^{-1} (Y)]_{A})))(f))\\
            &\qquad+a_{A}(X)(a_{A}(N(\phi_{A}^{-1}(Y)))(f))-a_{A}(Y)(a_{A}(N(\phi_{A}^{-1}(X)))(f))\\
            &\qquad -\varphi^{*}(a_{A}(N(\phi_{A}^{-1}([\phi_{A}^{-1}(X),\phi_{A}^{-1} (Y)]_{A})))(f))\\
            &=a_{A}([\phi_{A}^{-1}(NX), \phi_{A}^{-1}(Y)]_{A})(\varphi^{*}f)+a_{A}([\phi_{A}^{-1}(X), \phi_{A}^{-1}(NY)]_{A})(\varphi^{*}f)\\
            &\qquad -\varphi^{*}((\mbox{Ad}_{\varphi^{*}}^{-1}(a_{A}([\phi_{A}^{-1}(NX),\phi_{A}^{-1}(Y)]_{A})))(f))\\
            &\qquad -\varphi^{*}((\mbox{Ad}_{\varphi^{*}}^{-1}(a_{A}([\phi_{A}^{-1}(X),\phi_{A}^{-1}(NY)]_{A})))(f))\\
            &=\varphi^{*}((\varphi^{*})^{-1}(a_{A}([\phi_{A}^{-1}(NX),\phi_{A}^{-1}(Y)]_{A}))(\varphi^{*}f))\\
            &\qquad +\varphi^{*}((\varphi^{*})^{-1}(a_{A}([\phi_{A}^{-1}(X),\phi_{A}^{-1}(NY)]_{A}))(\varphi^{*}f))\\
            &\qquad -\varphi^{*}((\varphi^{*})^{-1}(a_{A}([\phi_{A}^{-1}(NX),\phi_{A}^{-1}(Y)]_{A}))(\varphi^{*}f))\\
            &\qquad -\varphi^{*}((\varphi^{*})^{-1}(a_{A}([\phi_{A}^{-1}(X),\phi_{A}^{-1}(NY)]_{A}))(\varphi^{*}f))\\
            &=0,
\end{align*}
where we use (v) in Lemma \ref{phi-N-lemma} and (ii) in Definition \ref{Hom-Lie algebroid-def}.
 \end{proof}
 
 \subsection{Hom-Poisson-Nijenhuis structures}\label{sub Hom-Poisson-Nijenhuis structures}
 
In this subsection, we introduce the notion of Hom-Poisson-Nijenhuis structures and prove the relation between Hom-Poisson-Nijenhuis structures and Hom-Lie bialgebroids and the existence of the hierarchy of Hom-Poisson-Nijenhuis structures, where these are the generalizations of properties of Poisson-Nijenhuis structures \cite{KMa}, \cite{K3} and main results in this paper.

 \begin{defn}
 Let $(A, \varphi,\phi_{A},[\cdot,\cdot]_{A},a_{A})$ be a Hom-Lie algebroid over $M$, $\pi $ a Hom-Poisson structure on $M$ and $N$ a Hom-Nijenhuis structure on $M$. Then $\pi$ and $N$ are {\it compatible} if they satisfy
\begin{eqnarray}
N\circ \pi ^\sharp =\pi ^\sharp \circ N^*, \label{kakansei}
\end{eqnarray}
and $C_\pi ^N$ given by
\begin{eqnarray}
C_\pi ^N(\alpha , \beta ):=[\alpha ,\beta ]_{N\pi ^\sharp }-[\alpha ,\beta ]_\pi ^{N^*} \label{compatibility hantei}
\end{eqnarray}
for any $\alpha $ and $\beta $ in $\Gamma (A^*)$ vanishes. Here we define
\begin{align}
[\alpha ,\beta ]_{N\pi ^\sharp }&:=\mathcal{L}_{N\pi ^\sharp \alpha }^{A}\beta -\mathcal{L}_{N\pi ^\sharp \beta }^{A}\alpha -d_{A}\langle N\pi ^\sharp \alpha ,\beta \rangle, \\
[\alpha ,\beta ]_\pi ^{N^*}&:=[N^*\alpha , \beta ]_\pi +[\alpha ,N^*\beta ]_\pi -N^*[\alpha ,\beta ]_\pi .
\end{align}
for any $\alpha $ and $\beta $ in $\Gamma (A^*)$. A pair $(\pi ,N)$ is a {\it Hom-Poisson-Nijenhuis structure} on $A$ if $\pi $ and $N$ is compatible.
 \end{defn}
 
 We also describe as $d_{N}$ and $\mathcal{L}^{N}$ the operators defined by the formula (\ref{A-gaibibun}), (\ref{A-Lie derivative}) and (\ref{A-Lie derivative extension}) for any bundle map $N:A\longrightarrow A$ respectively.
 
By using the $\phi_{A}$-invariances for a $2$-section and a bundle map, we have the following proposition: 
 
\begin{prop}\label{P-N structure equivalent condition}
 Let $(A, \varphi,\phi_{A},[\cdot,\cdot]_{A},a_{A})$ be a Hom-Lie algebroid over $M$, $\pi $ a $\phi_{A}$-invariant $2$-section in $\Gamma (\Lambda ^{2}A)$ and $N:A\longrightarrow A$ a $\phi_{A}$-invariant bundle map over $M$. We assume that $\pi^{\sharp}\circ N^{*}=N\circ \pi ^{\sharp}$. Then the following conditions are equivalent:
 \begin{enumerate}
\item[(i)]$C_{\pi}^{N}=0$; 
\item[(ii)]$[\cdot,\cdot]_{N,\pi}=[\cdot,\cdot]_{\pi_{N}}$;
\item[(iii)]$[\cdot,\cdot]_{N,\pi}=[\cdot,\cdot]_{\pi}^{N^{*}}$;
\item[(v)]For any $\alpha $ and $\beta $ in $\Gamma (A^{*})$,
\begin{align}
C_{(\pi,N)}'(\alpha,\beta):=(\mathcal{L}_{\pi^{\sharp}\alpha}^{A}N)^{*}&\phi_{A}^{\dagger}(\beta)-(\mathcal{L}_{\pi^{\sharp}\beta}^{A}N)^{*}\phi_{A}^{\dagger}(\alpha) \nonumber \\ 
                                                                      &+N^{*}d_{A}\langle \pi^{\sharp}\alpha ,\beta \rangle-d_{A}\langle \pi ^{\sharp }N^{*}\alpha,\beta \rangle
\end{align}
vanishes. 
\end{enumerate}
Here the bracket $[\cdot,\cdot]_{N,\pi}$ is defined by
\begin{align}
[\alpha ,\beta ]_{N,\pi}&:=\mathcal{L}_{\pi ^\sharp \alpha }^{N}\beta -\mathcal{L}_{\pi ^\sharp \beta }^{N}\alpha -d_{N}\langle \pi ^\sharp \alpha ,\beta \rangle
\end{align}
for any $\alpha $ and $\beta $ in $\Gamma (A^*)$. 
 \end{prop}

 
The following theorem is the generalization of the {\it hierarchy of Poisson-Nijenhuis structures} \cite{MM2}, \cite{V}:

\begin{theorem}
 Let $(A, \varphi,\phi_{A},[\cdot,\cdot]_{A},a_{A})$ be a Hom-Lie algebroid over $M$, $(\pi ,N)$ a Hom-Poisson-Nijenhuis structure on $A$. We set $\pi _0:=\pi $ and define a $2$-section $\pi _{k+1}$ by the condition $\pi _{k+1}^\sharp  =N\circ \pi _k^\sharp $ inductively. Then all pairs $(\pi _k,N^p)$ ($k,p\geq 0$) are Hom-Poisson-Nijenhuis structures on $A$. Furthermore for any $k,l\geq 0$, it follows that $[\pi _k,\pi _l]_{A}=0$. The set of Hom-Poisson-Nijenhuis structures $\{(\pi _k,N^p)\}$ is called the {\it hierarchy of Hom-Poisson-Nijenhuis structures} of $(\pi ,N)$.
 \end{theorem}

\begin{proof}
By Proposition \ref{P-N structure equivalent condition} and analogical calculations of Theorem $1.3$ in \cite{V}, this theorem holds.
 \end{proof}

The following result is the generalization of the result for Poisson-Nijenhuis structures by Kosmann-Schwarzbach \cite{K3}:

\begin{theorem}
 Let $(A, \varphi,\phi_{A},[\cdot,\cdot]_{A},a_{A})$ be a Hom-Lie algebroid over $M$, $\pi $ a Hom-Poisson structure on $A$ and $N$ a Hom-Nijenhuis structure on $M$. We assume that the set $\{d_{A}f\ |\ f\in C^{\infty}(M)\}$ generates $\Gamma (A^{*})$ as a $C^{\infty }(M)$-module. Then the following three conditions are equivalent:
 \begin{enumerate}
\item[(i)]$(\pi ,N)$ is a Hom-Poisson-Nijenhuis structure; 
\item[(ii)]$(A_{N},A_{\pi}^{*})$ is a Hom-Lie bialgebroid;
\item[(iii)]$(A_{\pi}^{*},A_{N})$ is a Hom-Lie bialgebroid.
\end{enumerate}
 \end{theorem}

\begin{proof}
By Proposition \ref{Hom-Lie bilagebroid gyaku ok}, the equivalence of (ii) and (iii) holds. We shall prove the equivalent of (i) and (iii). 
We set for any $\xi_1$ and $\xi_2$ in $\Gamma (\Lambda ^{*}\varphi^{!}T^{*}M)$,
\begin{eqnarray}\label{A_N}
A_{N,\pi }(\xi_1,\xi_2):=d_N[\xi _1 , \xi_2 ]_\pi -[d_N\xi _1, \xi _2 ]_\pi -(-1)^{\mathrm{deg}\xi _1+1}[\xi _1, d_N\xi_2 ]_\pi.
\end{eqnarray}
Then by straightforward calculations, for any $f,g$ in $C^\infty (M)$, $\alpha ,\beta $ and $\gamma $ in $\Gamma (\Lambda ^{*}\varphi^{!}T^{*}M)$, we obtain
\begin{align}
A_{N,\pi }(f,g)&=\langle d_{A}(\varphi^{*}f),(N\pi ^\sharp -\pi ^\sharp N^*)(d_{A}(\varphi^{*}g))\rangle ,\\
A_{N,\pi }(d_{A}f,g)&=C_N^\pi(d_{A}(\varphi^{*}f),d_{A}g),\\
A_{N,\pi }(d_{A}f,d_{A}g)&=-d_{A}(C_N^\pi(d_{A}f,d_{A}g)),\\
A_{N,\pi }(\alpha ,\beta \wedge \gamma )&=A_{N,\pi }(\alpha ,\beta) \wedge (\mbox{Ad}_{\varphi^{*}}^{\dagger})^{2}(\gamma )\nonumber \\
&\qquad +(-1)^{\mbox{deg}\alpha\mbox{deg}\beta }(\mbox{Ad}_{\varphi^{*}}^{\dagger})^{2}(\beta )\wedge A_{N,\pi }(\alpha ,\gamma ),\label{ANpi wedge 越え}\\
A_{N,\pi }(\alpha ,\beta )&=-(-1)^{(\mbox{deg}\alpha -1)(\mbox{deg}\beta  -1)}A_{N,\pi }(\beta ,\alpha ),
\end{align}
so that the conclusion follows from these equations and the assumption that the set $\{d_{A}f\ |\ f\in C^{\infty}(M)\}$ generates $\Gamma (A^{*})$ as a $C^{\infty }(M)$-module. 
 \end{proof}
 
 \begin{rem}
 Considering ``Hom-versions'' of several generalizations of Poisson-Nijenhuis structures, for example, Poisson quasi-Nijenhuis structures \cite{SX}, pseudo-Poisson Nijenhuis structures \cite{N2}, and so on, is very interesting.
 \end{rem}

\section{Hom-Dirac structures on Hom-Courant algebroids}\label{Hom-Dirac structures on Hom-Courant algebroids}

In this section, we introduce the notion of Hom-Dirac structures on a Hom-Courant algebroid and prove the relation between Hom-Dirac structures and Hom-Lie algebroids. We also prove a necessary and sufficient condition for the graph of a bundle map from $A^{*}$ into $A$ to be a Hom-Dirac structure on the Hom-Courant algebroid $A\oplus A^{*}$, where $(A,A^{*})$ is a Hom-Lie bialgebroid. These are the generalizations of properties of Dirac structures on a Courant algebroid and also main results in this paper.

\begin{defn}
Let $(E,\varphi, \phi_{E},\langle \! \langle\cdot,\cdot\rangle \! \rangle, \odot_{E},\rho_{E})$ be a Hom-Courant algebroid over $M$ and $[\![\cdot,\cdot ]\!]_{E}$ the Hom-Courant bracket on $\Gamma (E)$. A subbundle $L$ of $E$ is {\it isotropic} if it is isotropic under the pairing $\langle \! \langle \cdot ,\cdot \rangle \! \rangle $. A subbundle $L$ is {\it integrable} if $\Gamma (L)$ is closed under the bracket $[\![\cdot ,\cdot ]\!]_{E}$. A subbundle $L$ is {\it $\phi_{E}$-invariant} if it satisfies that $\phi_{E}(\Gamma (L))\subset \Gamma (L)$. A subbundle $L$ is a {\it Hom-Dirac structure} or a {\it Hom-Dirac subbundle} if it is maximally isotropic, integrable and $\phi_{E}$-invariant.
\end{defn}

The relationship between Dirac structures and Lie algebroids \cite{LWX} is generalized that between Hom-Dirac structures and Hom-Lie algebroids.

\begin{prop}
Let $(E,\varphi, \phi_{E},\langle \! \langle\cdot,\cdot\rangle \! \rangle, \odot_{E},\rho_{E})$ be a Hom-Courant algebroid over $M$, $[\![\cdot,\cdot ]\!]_{E}$ the Hom-Courant bracket on $\Gamma (E)$ and a subbundle $L$ a Hom-Dirac structure. Then $(L,\varphi, \phi_{E}|_{L},[\![\cdot ,\cdot ]\!]|_L,$ $\rho |_L)$ is a Hom-Lie algebroid over $M$.
\end{prop}

\begin{proof}
 Let $e_{1},e_{2}$ and $e_{3}$ be elements in $\Gamma (L)$ and $f$ a smooth function on $M$. First of all, we obtain $\phi_{E}(fe_{1})=\varphi^{*}f\cdot\phi_{E}(e_{1})$. Since $L$ is an isotropic, integrable and $\phi_{E}$-invariant subbundle in $E$, it follows that $[\![e_{1},e_{2}]\!]_{E}$ and $\phi_{E}(e_{3})$ are in $\Gamma (L)$ and that $\langle \! \langle [\![e_{1},e_{2}]\!]_{E}, \phi_{E}(e_{3}) \rangle \! \rangle=0$. Therefore by Proposition \ref{Jacobi-modoki}, we obtain the Hom-Jacobi identity for $[\![\cdot,\cdot ]\!]_{E}$. That $(\Gamma (E), \odot _{E}, \phi_{E})$ is Hom-Leibniz algebra means that $\phi_{E}$ is an algebra homomorphism with respect to $[\![\cdot,\cdot ]\!]_{E}$. Then, we calculate
 \begin{align*}
2[\![e_{1},fe_{2}]\!]_{E}&=(\varphi^{*}f\cdot e_{1}\odot _Ee_{2}+\rho_{E}(\phi_{E}(e_{1}))(f)\phi_{E}(e_{2}))\\
                         &\quad -(\varphi^{*}f\cdot e_{2}\odot _Ee_{1}-\rho_{E}(\phi_{E}(e_{1}))(f)\phi_{E}(e_{2})+\mathcal{D}f\cdot \varphi^{*}\langle \! \langle e_{1},e_{2}\rangle \! \rangle)\\
                         &=2(\varphi^{*}f\cdot [\![e_{1},fe_{2}]\!]_{E}+\rho_{E}(\phi_{E}(e_{1}))(f)\phi_{E}(e_{2}))
\end{align*} 
by Proposition \ref{function multiplication for odot}. From (ii) and (iii) in Definition \ref{Hom-Courant algebroid def}, we show that the condition (ii)' in Remark \ref{def iikae} is satisfied. From the above, $(L,\varphi, \phi_{E}|_{L},[\![\cdot ,\cdot ]\!]|_L,$ $\rho |_L)$ is a Hom-Lie algebroid.
 \end{proof}
 
 Let $(A,A^{*})$ be a Hom-Lie bialgebroid over $M$ and $H:A^{*}\longrightarrow A$ a bundle map over $M$. We set
 \begin{align*}
\mathrm{graph}(H):=\{H\xi+\xi\ |\ \xi\in A^{*}\}\subset A\oplus A^{*}.
\end{align*}

\begin{theorem}
 With the above notations, the graph $\mathrm{graph}(H)$ of a bundle map $H$ 
is a Hom-Dirac structure on $A\oplus A^{*}$ equipped with the Hom-Courant algebroid structure in Theorem \ref{Hom-Lie bi to Hom-Courant} if and only if there exists a $\phi_{A}$-invariant $2$-section $\pi $ in $\Gamma (\Lambda ^{2}A)$ such that $H=\pi ^{\sharp}$ and that the {\it Marer-Cartan type equation} holds: 
 \begin{align}\label{Marer-Cartan type equation}
d_{A^{*}}\pi +\frac{1}{2}[\pi ,\pi]_{A}=0.
\end{align}
 \end{theorem}

\begin{proof}
For any $\xi$ and $\eta $ in $\Gamma (A^{*})$, $H\xi+\xi$ and $H\eta +\eta$ are any elements in $\Gamma (\mbox{graph}(H))$. Since
 \begin{align*}
 2\langle \! \langle H\xi+\xi,H\eta+\eta \rangle \! \rangle=\langle \xi,H\eta\rangle +\langle \eta,H\xi\rangle,
\end{align*}
 we see that $\mbox{graph}(H)$ is isotropic if and only if there exists a $2$-section $\pi $ in $\Gamma (\Lambda ^{2}A)$ such that $H=\pi^{\sharp}$.  The maximality of $\mbox{graph}(H)$ is obvious. An element $\phi_{E}(H\xi+\xi)=\phi_{A}(H\xi)+\phi_{A}^{\dagger}(\xi)$ is in $\Gamma (\mbox{graph}(H))$ if and only if $\phi_{A}\circ H=H\circ \phi_{A}^{\dagger}$. Therefore by Proposition \ref{pi invariance and commutativity}, $\mbox{graph}(H)$ is maximally isotropic and $\phi_{E}$-invariant if and only if there exists a $\phi_{A}$-invariant $2$-section $\pi $ in $\Gamma (\Lambda ^{2}A)$ such that $H=\pi ^{\sharp}$. Then by formula (\ref{Hom-Courant bracket for Hom-Lie bi}), we compute
 \begin{align}\label{bracket culculus}
[\![\pi^{\sharp }\xi&+\xi,\pi^{\sharp}\eta+\eta ]\!]_{E}\nonumber\\
                    &=\left([\pi^{\sharp }\xi,\pi^{\sharp}\eta]_{A}+\mathcal{L}_{\xi}^{A^{*}}(\pi^{\sharp}\eta)-\mathcal{L}_{\eta}^{A^{*}}(\pi^{\sharp}\xi)+\frac{1}{2}d_{A^{*}}(\langle \pi^{\sharp }\xi,\eta \rangle-\langle \pi^{\sharp }\eta,\xi \rangle)\right)\nonumber\\
                    &\qquad +\left([\xi,\eta]_{A^{*}}+\mathcal{L}_{\pi^{\sharp}\xi}^{A}\eta-\mathcal{L}_{\pi^{\sharp}\eta}^{A}\xi-\frac{1}{2}d_{A}(\langle \pi^{\sharp }\xi,\eta \rangle-\langle \pi^{\sharp }\eta,\xi \rangle)\right)\nonumber\\
                    &=\left(\frac{1}{2}[\pi,\pi]_{A}(\phi_{A}^{\dagger}(\xi),\phi_{A}^{\dagger }(\eta),\cdot)+\pi^{\sharp}[\xi,\eta]_{\pi}\right.\nonumber\\
                    &\qquad \quad+\mathcal{L}_{\xi}^{A^{*}}(\pi^{\sharp}\eta)-\mathcal{L}_{\eta}^{A^{*}}(\pi^{\sharp}\xi)+d_{A^{*}}\langle \pi^{\sharp }\xi,\eta \rangle)+([\xi,\eta]_{A^{*}}+[\xi,\eta]_{\pi}),
\end{align}
where we use Proposition \ref{pi pi calculus} and the definition (\ref{pi bracket}) of $[\cdot,\cdot]_{\pi}$. 

On the other hand, for any $\xi,\eta,\zeta$ in $\Gamma (A^{*})$ and $\phi_{A}$-invariant $2$-section $\pi $, we calculate
 \begin{align*}
(d_{A^{*}}&\pi)(\phi_{A}^{\dagger}(\xi),\phi_{A}^{\dagger}(\eta),\zeta)\\
          &=a_{A^{*}}(\phi_{A}^{\dagger}(\xi))(\pi ((\phi_{A}^{\dagger})^{-1}(\phi_{A}^{\dagger}(\eta)),(\phi_{A}^{\dagger})^{-1}(\zeta)))\\
          &\quad -a_{A^{*}}(\phi_{A}^{\dagger}(\eta))(\pi ((\phi_{A}^{\dagger})^{-1}(\phi_{A}^{\dagger}(\xi)),(\phi_{A}^{\dagger})^{-1}(\zeta)))\\
          &\quad +a_{A^{*}}(\zeta)(\pi ((\phi_{A}^{\dagger})^{-1}(\phi_{A}^{\dagger}(\xi)),(\phi_{A}^{\dagger})^{-1}(\phi_{A}^{\dagger}(\eta))))\\
          &\quad -\phi_{A}(\pi)([(\phi_{A}^{\dagger})^{-1}(\phi_{A}^{\dagger}(\xi)),(\phi_{A}^{\dagger})^{-1}(\phi_{A}^{\dagger}(\eta))]_{A^{*}},\zeta)\\
          &\quad +\phi_{A}(\pi)([(\phi_{A}^{\dagger})^{-1}(\phi_{A}^{\dagger}(\xi)),(\phi_{A}^{\dagger})^{-1}(\zeta)]_{A^{*}},\phi_{A}^{\dagger}(\eta))\\
          &\quad -\phi_{A}(\pi)([(\phi_{A}^{\dagger})^{-1}(\phi_{A}^{\dagger}(\eta)),(\phi_{A}^{\dagger})^{-1}(\zeta)]_{A^{*}},\phi_{A}^{\dagger}(\xi))\\
          &=a_{A^{*}}(\phi_{A}^{\dagger}(\xi))(\langle \pi ^{\sharp}\eta,(\phi_{A}^{\dagger})^{-1}(\zeta)\rangle)-a_{A^{*}}(\phi_{A}^{\dagger}(\eta))(\langle \pi ^{\sharp}\xi,(\phi_{A}^{\dagger})^{-1}(\zeta)\rangle)\\
          &\quad +a_{A^{*}}(\zeta)(\langle \pi^{\sharp} \xi,\eta\rangle)-\pi([\xi,\eta]_{A^{*}},\zeta)\\
          &\quad -\pi(\phi_{A}^{\dagger}(\eta),[\xi,(\phi_{A}^{\dagger})^{-1}(\zeta)]_{A^{*}})+\pi(\phi_{A}^{\dagger}(\xi),[\eta,(\phi_{A}^{\dagger})^{-1}(\zeta)]_{A^{*}})\\
          &=\langle \mathcal{L}_{\xi}^{A^{*}}(\pi^{\sharp }\eta)-\mathcal{L}_{\eta}^{A^{*}}(\pi^{\sharp }\xi)+d_{A^{*}}\langle \pi^{\sharp}\xi,\eta \rangle -\pi^{\sharp }[\xi,\eta]_{A^{*}},\zeta\rangle,
\end{align*}
so that we obtain
 \begin{align}\label{d pi culculus}
(d_{A^{*}}\pi)&(\phi_{A}^{\dagger}(\xi),\phi_{A}^{\dagger}(\eta),\cdot)+\pi^{\sharp }[\xi,\eta]_{A^{*}}\nonumber\\
              &=\mathcal{L}_{\xi}^{A^{*}}(\pi^{\sharp }\eta)-\mathcal{L}_{\eta}^{A^{*}}(\pi^{\sharp }\xi)+d_{A^{*}}\langle \pi^{\sharp}\xi,\eta \rangle.
\end{align}
By substituting (\ref{d pi culculus}) in the equation (\ref{bracket culculus}), we have
 \begin{align}
[\![\pi^{\sharp }\xi&+\xi,\pi^{\sharp}\eta+\eta ]\!]_{E}\nonumber\\
                    &=\left(\left(d_{A^{*}}\pi+\frac{1}{2}[\pi,\pi]_{A}\right)(\phi_{A}^{\dagger}(\xi),\phi_{A}^{\dagger }(\eta),\cdot)+\pi^{\sharp}([\xi,\eta]_{A^{*}}+[\xi,\eta]_{\pi})\right)\nonumber\\
                                                       &\qquad +([\xi,\eta]_{A^{*}}+[\xi,\eta]_{\pi}).
\end{align}
Therefore $\mbox{graph}(\pi^{\sharp})$ is closed under $[\![\cdot,\cdot ]\!]_{E}$ if and only if
 \begin{align}
d_{A^{*}}\pi+\frac{1}{2}[\pi,\pi]_{A}=0
\end{align}
since $\phi_{A}$ is an isomorphism.

From the above, the proof is complete.
\end{proof}
 
\begin{example}
Let $(A, \varphi,\phi_{A},[\cdot,\cdot]_{A},a_{A})$ be a Hom-Lie algebroid over $M$ and $\pi $ a Hom-Poisson structure on $A$. Applying Theorem \ref{Hom-Lie bi to Hom-Courant} for the Hom-Lie bialgebroid $(A,A^*)$ in Example \ref{trivial example}, we see that the graph of $\pi ^\sharp$ satisfies the Maurer-Cartan type equation. Therefore $\mbox{graph}(\pi^\sharp)$ is Hom-Dirac structure on $A\oplus A^*$.
\end{example}

\end{document}